\documentclass[a4paper,11pt,reqno]{article}
\usepackage{tikz-cd} 
\usepackage{tikz}
\usepackage{relsize}
\usepackage{cite}
\usepackage{color}
\usepackage{amsmath,amsthm,amssymb}
\usepackage{enumitem}
\usepackage{algorithm}
\usepackage{algpseudocode}
\usepackage[margin=2.6cm]{geometry}
\usepackage{mathdots}

\makeatletter
\usepackage{comment}
\let\wfs@comment@comment\comment
\let\comment\@undefined

\usepackage{changes}
\let\wfs@changes@comment\comment
\let\comment\@undefined

\newcommand\comment{%
    \ifthenelse{\equal{\@currenvir}{comment}}
    {\wfs@comment@comment}
    {\wfs@changes@comment}%
}
\makeatletter
\def\namedlabel#1#2{\begingroup
    #2%
    \def\@currentlabel{#2}%
    \phantomsection\label{#1}\endgroup
}
\makeatother

\usepackage{scalerel,stackengine}
\stackMath
\usepackage{amsmath,amssymb,amsthm,mathrsfs,amsfonts,dsfont,stmaryrd} 
\usepackage{multirow}
\usepackage{amscd}   
\usepackage{fullpage}
\usepackage[all]{xy}  
\usepackage{mathrsfs,enumitem}
\usepackage[titletoc,toc,title]{appendix}
\usepackage{xcolor}  
\usepackage[normalem]{ulem}
\usepackage{diagbox}
\usepackage{caption}
\usepackage[colorlinks = true,
            linkcolor = black,
            urlcolor  = blue,
            citecolor = red,
            anchorcolor = blue]{hyperref}
\usepackage{graphicx}
 \usepackage{color}
 \usepackage{tikz, framed}
\usetikzlibrary{arrows,shapes,matrix,decorations.pathmorphing}
\usetikzlibrary{backgrounds}
\usepackage[english]{babel} 

\theoremstyle{definition}
\newtheorem{theorem}{Theorem}[section]
\newtheorem{definition}[theorem]{{{Definition}}}
\newtheorem{example}[theorem]{{{Example}}}
\newtheorem{notation}[theorem]{{{Notation}}}
\newtheorem{remark}[theorem]{{{Remark}}}

\newtheorem{corollary}[theorem]{{{Corollary}}}
\newtheorem{proposition}[theorem]{{{Proposition}}}
\newtheorem{lemma}[theorem]{{{Lemma}}}

\newcommand{\mL}{\mathcal{L}}

\newcommand{\mP}{\mathcal{P}}

\newcommand{\mA}{\mathcal{A}}

\newcommand{\I}{\mathcal{I}}
\newcommand{\mH}{\mathcal{H}}

\newcommand{\mZ}{\mathcal{Z}}

\newcommand{\len}{\textup{len}}
\newcommand{\h}{\textup{h}}

 \newcommand{\RR}{\mathbb{R}}

 \newcommand{\cC}{\textup{\bf{C}}}
 
\newcommand{\F}{\mathbb{F}}

\newcommand{\one}{{\mathbf{1}}}
\newcommand{\zero}{{\mathbf{0}}}

\DeclareMathOperator{\cyc}{cyc}
\DeclareMathOperator{\cl}{cl}
\providecommand{\keywords}[1]
{
\noindent\textbf{{\small \textbf{Keywords.}}} #1
}

\title{The Cyclic Flats of an $\mL$-Polymatroid}

\usepackage{authblk}
\author[1]{Eimear Byrne}
\affil[1,2]{School of Mathematics and Statistics, University College Dublin, Belfield, Ireland}
\author[2]{Andrew Fulcher}

\begin{document}
\maketitle

\begin{abstract}
   We consider structural properties of $\mL$-polymatroids, especially those defined on a finite complemented modular lattice $\mL$. We introduce a set of cover-weight axioms and establish a cryptomorphism between these axioms and the rank axioms of an $\mL$-polymatroid. We introduce the notion of a cyclic flat of an $\mL$-polymatroid and study properties of its lattice of cyclic flats. 
   We show that the weighted lattice of cyclic flats of an $\mL$-polymatroid 
   $\mP$, along with the atomic weights of $\mP$, is sufficient to define its rank function on $\mL$. In our main result, we characterize those weighted lattices $(\mZ,\lambda)$ such that $\mZ\subseteq\mathcal{L}$ is the collection of cyclic flats of an $\mL$-polymatroid.   
\end{abstract}



\tableofcontents


\section{Introduction}
The lattice of cyclic flats of a matroid plays an important role in matroid theory. Every finite lattice is isomorphic to the lattice of cyclic flats of some matroid \cite{bonin_demier,sims}, and transversal matroids can be characterized in terms of their cyclic flats and their ranks \cite{bkdm,bryl_transv,ingleton,mason}. 
There have also been some interesting studies of invariants of matroids and polymatroids in relation to their cyclic flats. For example, in \cite{eberhardt}, an expression for the Tutte polynomial of a matroid is given in terms of the isomorphism class of the lattice of cyclic flats, along with their sizes and ranks. The ${\cal{G}}$-invariant of a matroid and its connection with the cyclic flats of a matroid were considered in \cite{bonin_Ginv,bonin_kung_Ginv}. More precisely, it was shown that these invariants can be computed from the configuration of a matroid, which carries strictly less information than its weighted lattice of cyclic flats.

A matroid can be determined by its lattice of cyclic flats, along with their ranks. Specifically, knowledge of the lattice of cyclic flats, along with the ranks of its elements, is sufficient to reconstruct the rank function of the matroid. However, the lattice of cyclic flats of a matroid has far fewer elements than its lattice of flats, providing a more compact representation of the matroid. See \cite{fghw} for an explicit algorithm to reconstruct the lattice of flats of a matroid from its lattice of cyclic flats. The results of \cite{bonin_demier,sims} have also been extended to more general settings. In \cite{Csirm20}, it is shown that a polymatroid is completely determined by its weighted lattice of cyclic flats, along with the ranks of its singleton sets. Additionally, \cite{AlfByr22} provides a cryptomorphism for cyclic flats in the context of $q$-matroids.

In this paper, we consider the theory of polymatroids in the context of more general lattices, with the main focus on finite complemented modular lattices, which comprise a more general class than subset and subspace lattices. More precisely, any finite complemented modular lattice is a direct product of a Boolean lattice and projective geometries \cite[Chapter 4, Section 7]{birkhoff_lattice}. We define an $\mL$-polymatroid as a non-negative, bounded, increasing, and submodular function on a lattice $\mL$. We introduce the concept of a cyclic flat of an $\mL$-polymatroid when $\mathcal{L}$ is complemented and modular. Our main contribution is to extend the results of \cite{AlfByr22,bonin_demier,Csirm20} to the more general setting of an $\mL$-polymatroid. That is, we derive a cryptomorphism for $\mL$-polymatroids in terms of a weighted lattice of cyclic flats $(\mZ,\lambda)$, along with the ranks of the atoms of $\mL$. To achieve this, we propose a list of cyclic flat axioms that a weighted lattice $(\mZ,\lambda)$ must satisfy in order to be embedded into the weighted lattice $(\mL,\rho)$ of an $\mL$-polymatroid as its lattice of cyclic flats.  

We remark that generalizations of matroids and polymatroids have been in the literature for some time. In \cite{faigle}, matroid theory was extended from Boolean lattices to partially ordered sets. There have been several papers on supermatroids, in which matroid theory on various classes of lattices is considered, such as distributive and modular lattices \cite{Dunstan,MAEHARA_supermat,tardos_supermat,WILD_supermat}. More recently, $q$-analogues of matroids and polymatroids have been studied \cite{JP18,BCJ_JCTB,bcij,gpr_q_complex,gorla2019rank,shiromoto19}. 
However, there are yet relatively few works on generalizations of polymatroids to arbitrary finite complemented modular lattices; see \cite{alfarano2023critical,GorlaSalizz_Latroid,Vertigan2004Latroids}. 

This paper is organized as follows. In Section \ref{sec:prelim}, we establish notation and give some fundamentals. In Section \ref{sec:covlatt}, we introduce the cover-weight axioms and show they are equivalent to the rank axioms of an $\mL$-polymatroid, giving a cryptomorphism that extends results of \cite{BCJ17,BF22}. In Section \ref{sec:cyclicflat}, we introduce the notion of a cyclic element of an $\mL$-polymatroid. Moreover, we describe the lattice of cyclic flats of an $\mathcal{L}$-polymatroid. We show that the rank function of an $\mL$-polymatroid is completely determined by its weighted lattice of cyclic flats, along with the ranks of the atoms of $\mL$. We introduce the concept of a weakly decomposable function, which is crucial for capturing the behavior of the rank function of an $\mL$-polymatroid in terms of its weighted lattice of cyclic flats. Section \ref{sec:main} contains the main contribution of this paper. We introduce a set of six cyclic flat axioms and demonstrate a cryptomorphism for $\mL$-polymatroids. In Section \ref{sec:comments}, we close the paper with some commentary on the cyclic flat axioms given in Section \ref{sec:main} and their relations to cyclic flat axioms of matroids, polymatroids, and $q$-matroids.

\section{Preliminaries}\label{sec:prelim}
The primary difference between polymatroids and their $q$-analogues is determined by the underlying lattices upon which they are defined; the former is defined on the power set of a finite set, and the latter is defined on the subspace lattice of a finite-dimensional vector space over a finite field. Both of these lattices are \emph{complemented modular} lattices, but in the case of the Boolean lattice, complements are unique. In terms of the underlying lattice properties, many of the results we present in this paper only rely on the lattice being complemented and modular. For background reading on lattice theory, see \cite{birkhoff_lattice,gratz_lattice,romanlattices}. We include the following fundamental definitions and notation.

\begin{definition}
Let $\mathcal{L}$ be a finite lattice with meet $\wedge$, join $\vee$, and partial order $\leq$. We write $\one=1_\mathcal{L}$ and $\zero=0_\mathcal{L}$.
Let $A, B\in\mathcal{L}$.
\begin{enumerate}
    \item If $A \wedge B = \zero$, we write $A \dot \vee B = A \vee B$. 
    \item An {\em interval} $[A,B]\subseteq\mathcal{L}$ is the set of all $X\in\mathcal{L}$ such that $A\leq X\leq B$. It defines the {\em interval sublattice} $([A,B],\leq,\vee,\wedge)$. We write $\I(\mL)$ to denote the collection of intervals of $\mL$. We define $(A,B)=\{C\in[A,B]:A< C < B\}$.
    \item Let $C\in  [A,B]$. We say that $D$ is a \emph{complement} of $C$ in $[A,B]$  if $C \wedge D = A$ and $C \vee D = B$. We define $\textbf{C}(A)=\{C \in \mathcal{L}: C\wedge A=\zero, C \vee A = \one \}$. The lattice $\mathcal{L}$ is {\em complemented} if every $C \in \mL$ has a complement in $\mL$. The lattice $\mL$ is {\em relatively complemented} every interval $[A,B]\subseteq \mathcal{L}$ is complemented.
    \item For $A<B$, if $X\in[A,B]$ implies that $X=A$ or $X=B$, then $B$ is a {\em cover} of $A$ and we write $A\lessdot B$. We also say that $A$ is \emph{covered} by $B$ or that $B$ {\em covers} $A$.
    \item An {\em atom} of $\mL$ is any element that covers $\textbf{0}$. A {\em coatom} of $\mL$ is any element that \textbf{1} covers. 
    We define $\mA([A,B])=\{ X \in [A,B]: A \lessdot X \}$ and 
    $\mH([A,B])=\{X \in [A,B] : X \lessdot B\}$. We also define
    $\mA(B)=\mA([\zero,B])$ and $\mH(B)=\mH([\zero,B])$.
    \item 
    A {\em chain} from $A$ to $B$ is a totally ordered subset of $[A,B]$ with respect to $\leq$ that contains $A$ and $B$. A chain from $A$ to $B$ is {\em maximal} in $[A,B]$ if it is not properly contained in any other chain from $A$ to $B$. 
    A {\em chain} from $A$ to $B$ is written in the form
    $ A = X_0 < \cdots < X_k=B $, in which case we say that the chain has {\em length} $k$.
    \item We define the {\em height} of $B$, which we denote by $\h(B)$, to be the maximum length of all chains from \textbf{0} to $B$. The {\em length} of $[A,B]$ is defined to be $\len([A,B])=\h(B)-\h(A)$.
    \item Let $S \subseteq \mA(\mL)$ be non-empty and let $A$ be the join of the elements of $S$. We say that $S$ is {\em independent} if
    $\h(A)=|S|$.
    \item The lattice $\mL$ is {\em modular} if for all $A,B,C \in \mL$, we have that
    $A \geq C$ implies $(A \wedge B) \vee C = A \wedge (B \vee C).$ This property is referred to as the \emph{modular law}.
    \item The lattice $\mathcal{L}$ is \emph{atomic} if every element of $\mathcal{L}$ can be written as the join of atoms of $\mathcal{L}$.
    \item For lattices $\mathcal{L}_i$, $i\in [n]$, the \emph{direct product} $\mathcal{L}_1\times \mathcal{L}_2\times\cdots\times \mathcal{L}_n$ is the lattice with elements $(a_1,a_2,\dots,a_n)$, $a_i\in\mathcal{L}_i$ for each $i\in[n]$ and order defined by $(a_1,\dots,a_n)\leq (b_1,\dots,b_n)$ if and only if $a_i\leq b_i$ in $\mathcal{L}_i$ for each $i\in[n]$.
    \item For lattices $\mathcal{L}_1$ and $\mathcal{L}_2$ with respective joins $\vee_1$ and $\vee_2$, suppose that $\Phi:\mathcal{L}_1\to\mathcal{L}_2$ is a map. We say that $\Phi$ is a \emph{lattice isomorphism} when it is bijective and $\Phi(A\vee_1 B)=\Phi(A)\vee_2 \Phi(B)$ for every $A,B\in\mathcal{L}_1$. We then call $\mathcal{L}_1$ and $\mathcal{L}_2$ \emph{isomorphic} and write $\mathcal{L}_1\cong\mathcal{L}_2$.
\end{enumerate}
\end{definition}
We remark that the notation $\dot\vee$ is not standard. However we choose to include it to provide a reminder in some of the more technical proofs.
For the remainder of this paper, $\mL$ will denote a finite lattice with meet $\wedge$, join $\vee$, minimal element $\zero$, maximal element $\one$, and order $\leq$. 
If the lattice $\mL$ is modular, then all its maximal chains have the same length (the Jordan-Dedekind condition). 
Every complemented modular lattice is relatively complemented and atomic. We will make frequent use of these properties in this paper. Note that in a complemented modular lattice, a given element need not have a unique complement (for example, in a subspace lattice, elements have multiple complements).

We now recall fundamental results in lattice theory that underpin our use of complemented modular lattices. See {\cite[Chapter IV.6, Theorem 10]{birkhoff_lattice}} and {\cite[Chapter VIII.9, Theorem 20]{birkhoff_lattice}}.
\begin{theorem}\label{thm:compmod_is_direct_prod}
    \begin{enumerate}
        \item Any complemented modular lattice is the direct product of a Boolean lattice and projective geometries.
        \item Every finite atomic projective geometry of height greater than three is isomorphic to the lattice of subspaces of $\F^{m}$ for some finite field $\F$ and positive integer $m$.
    \end{enumerate}
\end{theorem}

We now recall the following well-known theorem for modular lattices.
\begin{theorem}[Isomorphism Theorem for Modular Lattices]\label{thm:isom_thm_lattices}
    Let $\mL$ be a modular lattice. Let $A,B\in\mathcal{L}$. Then the intervals $[A,A\vee B]$ and $[A\wedge B,B]$ are isomorphic.
\end{theorem}
An immediate consequence of Theorem~\ref{thm:isom_thm_lattices} is that if $\mL$ is modular, then for all $A,B\in\mathcal{L}$, we have $\textup{h}(A)+\textup{h}(B)=\textup{h}(A\vee B)+\textup{h}(A\wedge B)$. We will use this property frequently in later proofs.

We now state the definition of an $\mathcal{L}$-polymatroid, which naturally generalises the notion of a $q$-polymatroid to more general lattices.

\begin{definition}\label{def:q-polymatroid}
    An $\mL$-\emph{polymatroid} is a pair $\mP=(\mL,r)$ where $r:\mathcal{L}\to\mathbb{R}$ is a function satisfying the following properties for all $A,B\in\mL$.
    \begin{description}
        \item[\namedlabel{r1}{{\rm (R1)}}] $r(\zero)=0$,
        \item[\namedlabel{r2}{{\rm (R2)}}] if $A \leq B$, then $r(A) \leq r(B)$, i.e., $r$ is non-decreasing on $\mL$,
        \item[\namedlabel{r3}{{\rm (R3)}}] $r(A)+r(B)\geq r(A\vee B)+r(A\wedge B)$, i.e., $r$ is submodular on $\mL$.
    \end{description}
\end{definition}

If $\mL = \mL(\F_q^n)$ (the subspace lattice of the vector space $\mathbb{F}_q^n$), we say that $\mP$ is a $q$-polymatroid. If, furthermore, $r$ is integer-valued and $r(A)\leq\dim(A)$ for all $A\in\mathcal{L}$, we say that $\mP$ is a $q$-matroid. If instead $\mL$ is a Boolean (subset) lattice, then $\mP$ is a polymatroid, and if, furthermore, $r$ is integer-valued and $r(A)\leq|A|$ for all $A\in\mathcal{L}$, then $\mP$ is a matroid. Definition~\ref{def:q-polymatroid} directly extends the definition of a $q$-polymatroid given in \cite{shiromoto19} (see also \cite{gorla2019rank}).

The remainder of this section is a collection of foundational lattice theoretic results that will be of use later in this paper.

\begin{lemma}\label{lem:gen_latt_distrib}
    For any $A,B,C\in\mathcal{L}$, we have $(A\wedge B)\vee (A\wedge C)\leq A\wedge (B\vee C).$
\end{lemma}

\begin{lemma}\label{lem:dir_join_assoc}
    Let $\mL$ be a modular lattice.
    Let $A,B,C\in\mathcal{L}$ and suppose that $A\wedge B=\zero$ and $(A\vee B)\wedge C=\zero$. Then we have $A\wedge(B\vee C)=\zero$. 
\end{lemma}
\begin{proof}
    Since $(A\vee B)\wedge C=\zero$, we have $B\wedge C=\zero$. Moreover, since $A\wedge B=\zero$, along with the associativity of the join, we obtain  
    $$\textup{h}( A\vee B\vee C)=\textup{h}((A\dot\vee B)\dot\vee C)=\textup{h}(A)+\textup{h}(B)+\textup{h}(C)=\textup{h}(A)+\textup{h}(B\dot\vee C).$$
    Furthermore, since $\textup{h}(A\vee(B\vee C))=\textup{h}(A)+\textup{h}(B\vee C)-\textup{h}(A\wedge(B\vee C))$, it follows that $\textup{h}(A\wedge(B\vee C))=0$. 
\end{proof}

The following is an easy consequence of Theorem~\ref{thm:compmod_is_direct_prod}.
\begin{corollary}\label{cor:extend_to_complement}
    Let $\mL$ be a complemented modular lattice.
    Let $A,B\in\mathcal{L}$ and suppose that $A\wedge B=\zero$. Then there exists $A^c\in\mathbf{C}(A)$ such that $B\leq A^c$.
\end{corollary}

By Corollary~\ref{cor:extend_to_complement}, if $A,B\in\mathcal{L}$ and $A\wedge B=\zero$, then $B$ is a relative complement of $A$ in $[\zero,A \vee B]$. Therefore, we say that $B$ can be {\em extended} to a complement of $A$ in $\mL$.

\begin{definition}\label{def:decomposes}
    Let $A,B\in\mathcal{L}$ and let $B^c\in \mathbf{C}(B)$. 
    If $A=(A\wedge B)\dot\vee(A\wedge B^c)$, we say that $B^c$ \emph{decomposes} $A$; equivalently, $A\wedge B^c$ is a relative complement of $A\wedge B$ in the interval $[\zero,A]$.
    The set of elements of $\mathbf{C}(B)$ that decompose $A$ is denoted by 
    $$\mathbf{C}(B;A)=\{ B^c \in \mathbf{C}(B) : (A\wedge B) \dot\vee (A \wedge B^c) = A \}.$$ 
\end{definition}

Note that, in general, the property in Definition~\ref{def:decomposes} is not satisfied by all complements of a given element. However, we will frequently use elements with this property in subsequent sections.

\begin{proposition}\label{prop:compdecomp}
    Let $\mL$ be a complemented modular lattice.
    Let $B\leq A\in\mathcal{L}$. Then $B^c\in\mathbf{C}(B;A)$ for any $B^c\in \mathbf{C}(B)$.
\end{proposition}
\begin{proof}
    The statement follows by applying the assumption $A\leq B$ and the modular law to evaluate $(A\wedge B)\vee (A\wedge B^c)$.
\end{proof}

Lemma~\ref{lem:compl_partition} and Corollary~\ref{cor:decomp_both} clearly hold if $\mathcal{L}$ is either a subspace lattice or a Boolean lattice. We thus deduce them from Theorem~\ref{thm:compmod_is_direct_prod}.

\begin{lemma}\label{lem:compl_partition}
    Let $\mL$ be a complemented modular lattice, and 
    let $A,B\in\mL$. Then $\cC(B;A)$ is non-empty.
\end{lemma}

\begin{corollary}\label{cor:decomp_both}
    Let $\mL$ be a complemented modular lattice, and let $A,B,C\in\mathcal{L}$ such that $A\leq B$. Then $\mathbf{C}(C;A)\cap \mathbf{C}(C;B)$ is non-empty.
\end{corollary}


\section{The Cover-Weighted Lattice of an \texorpdfstring{$\mathcal{L}$}{L}-Polymatroid}\label{sec:covlatt}

Throughout this section, we let $(\mathcal{L},r)$ denote an arbitrary fixed $\mL$-polymatroid.

The notion of a weighted lattice is well established; see \cite{whittle}, for instance. We consider two similar objects, which we call an interval-weighted lattice and a cover-weighted lattice.
It was shown in \cite{BCJ17} that the rank function of a $q$-matroid is cryptomorphic to a bicoloring of the support lattice satisfying a set of axioms on its intervals of length two. This characterization of $q$-matroids is intuitive and useful, as demonstrated in \cite{BCJ17}, \cite{BF22}, and \cite{CJ22}. In this section, we 
provide a characterization of $\mL$-polymatroids in terms of interval-weight axioms and show that, in the case where $\mL$ is modular, there exist cover-weight axioms on the intervals of length two, yielding a cryptomorphism with the rank function of an $\mL$-polymatroid.

Recall that we denote the collection of intervals of $\mL$ by $\I(\mL)$. 

\begin{definition}
    An \emph{interval-weighted lattice} $(\mathcal{L},W)$ is a lattice $\mathcal{L}$ equipped with a function $W:\mathcal{I}(\mathcal{L})\rightarrow\mathbb{R}_{\geq0}$. We call $W$ an \emph{interval weighting} of $\mathcal{L}$.
\end{definition}

We now introduce axioms for an interval-weighted lattice, which we will show in 
Theorem~\ref{thm:q-poly_lattice_crypt} to yield a cryptomorphism for $\mL$-polymatroids, as defined in Definition \ref{def:q-polymatroid}.

\begin{definition}
    Let $(\mathcal{L},W)$ be an interval-weighted lattice. We define the \emph{interval-weight} axioms as follows.
    \begin{description}
        \item[\namedlabel{w1}{\rm (IW1)}] For all $[A,B]\in\mathcal{I}(\mathcal{L})$ and every chain $A=X_0<\dots< X_m=B$ ($m\in\mathbb{N}_0$), we have 
        $$\sum_{i=1}^mW([X_{i-1},X_i])=W([A,B]).$$
        \item[\namedlabel{w2}{\rm (IW2)}] For all $A, X \in\mathcal{L}$, we have $W([A\wedge X,A])\geq W([X,A\vee X])$.
    \end{description}
\end{definition}

As we will see in the next theorem, if $(\mL,W)$ satisfies the interval-weight axioms, then if $r_W(A)=W([\zero,A])$ for all $A\in\mathcal{L}$, then $(\mathcal{L},r_W)$ is an $\mathcal{L}$-polymatroid. The proof is a straightforward verification, so we omit it.

\begin{theorem}\label{thm:q-poly_lattice_crypt}
Let $(\mL,W)$ be an interval-weighted lattice, and let 
$r:\mathcal{L}\rightarrow\mathbb{R}_{\geq 0}$.
Define the following functions $W_r:\mathcal{I}(\mL) \longrightarrow \mathbb {R}$ and $r_W:\mathcal{L} \longrightarrow \mathbb{R}$, respectively by
     \[
     W_r([A,B])=r(B)-r(A) \text{ for all }[A,B] \in \I(\mL) \text{ and } 
     r_W(A)=W([\zero,A]) \text{ for all } A \in \mL.
     \]
If $(\mathcal{L}, r)$ is an $\mathcal{L}$-polymatroid, then $(\mathcal{L}, W_r)$ satisfies
the interval-weight axioms and $r_{W_r} = r$. Likewise, If $(\mathcal{L}, W )$ satisfies the interval-weight
axioms, then $(\mathcal{L}, r_W )$ is an $\mathcal{L}$-polymatroid and $W_{r_W} = W$.
\end{theorem}

If $\mL$ is a modular lattice, we may replace axioms \ref{w1} and \ref{w2} with axioms that refer only to intervals of length at most two in $\mL$. First, we introduce the notion of a cover weighting.

\begin{definition}
    A \emph{cover-weighted} lattice $(\mathcal{L},w)$ is a pair for which $\mathcal{L}$ lattice and $w$ is a non-negative real-valued function $w$ defined on intervals of length one in $\mathcal{L}$. We call $w$ a \emph{cover weighting} of $\mathcal{L}$.
\end{definition}

We now specialize to the case for which $\mL$ is a finite modular lattice.

\begin{definition}
Let $\mL$ be modular, equipped with a cover weighting $w$. We define the following {\em cover-weight axioms} for all intervals $[A,B]\in\mathcal{L}$ of length two and all elements $X,Y\in(A,B)$.
    \begin{description}
        \item[\namedlabel{cw1}{{\rm (CW1)}}] We have $w([A,X])+w([X,B])=w([A,Y])+w([Y,B])$.
        \item[\namedlabel{cw2}{{\rm (CW2)}}] If $X\neq Y$, then $w([A,X])\geq w([Y,B])$.
    \end{description}
\end{definition}

We now establish a cryptomorphism between the cover weight axioms  \ref{cw1} and \ref{cw2} and the interval weight axioms \ref{w1} and \ref{w2}.

\begin{theorem}\label{th:covwt}
    Let $\mL$ be modular, equipped with a cover-weighting $w$ that satisfies \ref{cw1} and \ref{cw2}. 
    The following hold.
    \begin{enumerate}
        \item 
        For any $[A,B] \in \I(\mL)$ and any maximal chain $A=X_0\lessdot\cdots\lessdot X_m=B$, the value 
        $\sum_{i=1}^mw([X_{i-1},X_i])$ is determined by $A$ and $B$. In particular,
        we have a well-defined function $W_w: \I(\mL) \longrightarrow \mathbb{R}_{\geq 0}$ determined by
        \[ 
           W_w([A,B])=\sum_{i=1}^mw([X_{i-1},X_i]) \text{ for all } [A,B] \in \I(\mL).
        \]
        \item The function $W_w$ is an interval weighting that satisfies the interval-weight axioms.
    \end{enumerate}
    Thus, $(\mL,r_{W_w})$ is an $\mL$-polymatroid, where $r_{W_w}(A)=W_w([\zero,A])$ for all $A \in \mL$.
\end{theorem}
    
\begin{proof}
    Let $[A,B]\in\mathcal{I}(\mathcal{L})$. If $\textup{len}([A,B])\leq 2$, then by \ref{cw1}, $\sum_{i=1}^2w([X_{i-1},X_i])$ is determined by $[A,B]$ and therefore $W_w([A,B])=\sum_{i=1}^2w([X_{i-1},X_i])$ satisfies \ref{w1} on $[A,B]$. For $X,Y\in\mathcal{L}$, if $\textup{len}([X\wedge Y,X\vee Y])\leq 2$, then by \ref{cw2}, $W_w$ satisfies \ref{w2} for $X$ and $Y$. We now proceed by induction on $\textup{len}([A,B])$ and $\textup{len}([X\wedge Y,X\vee Y])$.

    Let $[A,B]\in\mathcal{I}(\mathcal{L})$ have length $m$. Since $\mathcal{L}$ is modular, any two maximal chains $A=X_0\lessdot X_1\lessdot\dots\lessdot X_m=B$ and $A=Y_0\lessdot Y_1\lessdot\dots\lessdot Y_m=B$ have equal length. If $X_1=Y_1$, then there is nothing to show due to the induction hypothesis. Therefore, let $X_1,Y_1\lessdot Z=X_1\vee Y_1$. By the induction hypothesis, $W_w$ is both well-defined and satisfies the identity from \ref{w1} on $[A,Z]$, $[X_1,B]$, and $[Y_1,B]$. We therefore obtain
    $$W_w([A,X_1])+W_w([X_1,B])=W_w([A,Z])+W_w([Z,B])=W_w([A,Y_1])+W_w([Y_1,B]),$$
    from which it follows that $W_w$ is both well-defined and satisfies the identity from \ref{w1} on $[A,B]$, and thus, on any interval in $\mathcal{L}$.

    Let $X,Y\in\mathcal{L}$ such that $\textup{len}([X\wedge Y,X\vee Y])=m$. Let $V\in[X\wedge Y,X]$ cover $X\wedge Y$. By the modularity of $\mathcal{L}$, we deduce that $\len([X\wedge Y, V\vee Y]),\len([V,X\vee Y])<m$. By the induction hypothesis, we obtain $W_w([X\wedge Y,Y])\geq W_w([V,V\vee Y])\geq W_w([X,X\vee Y])$.
\end{proof}

As a direct consequence of Theorems \ref{thm:q-poly_lattice_crypt} and \ref{th:covwt}, we now have the following corollary, which gives a characterisation of an $\mL$-polymatroid in terms of the rank function on intervals of length 2 in $\mL$.

\begin{corollary}\label{cor:int2}
    Let $\mL$ be modular. Let $r:\mL \to \mathbb{R}_{\geq 0}$ be such that \ref{r1} and \ref{r2} hold, and \ref{r3} holds on every interval of length 2 in $\mL$. Then $(\mL,r)$ is an $\mL$-polymatroid.
\end{corollary}

We will use the characterization given by Corollary \ref{cor:int2} in the proof of Theorem \ref{thm:axioms_imply_q-poly}.

\begin{figure}
    \centering
    \scalebox{1.5}{
    \begin{tikzpicture}
        \node[] (0) at (0,0) {\tiny $0$};
        \node[] (1) at (0,3) {\tiny $\mathbb{F}_2^3$};
        \node[] (a) at (-3,1) {\tiny $100$};
        \node[] (b) at (-2,1) {\tiny $010$};
        \node[] (c) at (-1,1) {\tiny $110$};
        \node[] (d) at (0,1) {\tiny $111$};
        \node[] (e) at (1,1) {\tiny $011$};
        \node[] (f) at (2,1) {\tiny $001$};
        \node[] (g) at (3,1) {\tiny $101$};
        \node[] (h) at (-3,2) {\tiny $\begin{smallmatrix}100\\010\end{smallmatrix}$};
        \node[] (i) at (-2,2) {\tiny $\begin{smallmatrix}100\\011\end{smallmatrix}$};
        \node[] (j) at (-1,2) {\tiny $\begin{smallmatrix}100\\001\end{smallmatrix}$};
        \node[] (k) at (0,2) {\tiny $\begin{smallmatrix}010\\001\end{smallmatrix}$};
        \node[] (l) at (1,2) {\tiny $\begin{smallmatrix}101\\010\end{smallmatrix}$};
        \node[] (m) at (2,2) {\tiny $\begin{smallmatrix}101\\011\end{smallmatrix}$};
        \node[] (n) at (3,2) {\tiny $\begin{smallmatrix}110\\001\end{smallmatrix}$};
        \path [-,red] (0) edge (a);
        \path [-,red] (0) edge (b);
        \path [-,black] (0) edge (c);
        \path [-,red] (0) edge (d);
        \path [-,black] (0) edge (e);
        \path [-,black] (0) edge (f);
        \path [-,black] (0) edge (g);
        \path [-,red] (1) edge (h);
        \path [-,red] (1) edge (i);
        \path [-,green] (1) edge (j);
        \path [-,green] (1) edge (k);
        \path [-,red] (1) edge (l);
        \path [-,green] (1) edge (m);
        \path [-,green] (1) edge (n);
        \path [-,red] (a) edge (h);
        \path [-,red] (a) edge (i);
        \path [-,black] (a) edge (j);
        \path [-,red] (b) edge (h);
        \path [-,black] (b) edge (k);
        \path [-,red] (b) edge (l);
        \path [-,green] (c) edge (h);
        \path [-,red] (c) edge (m);
        \path [-,red] (c) edge (n);
        \path [-,red] (d) edge (i);
        \path [-,red] (d) edge (l);
        \path [-,black] (d) edge (n);
        \path [-,green] (e) edge (i);
        \path [-,red] (e) edge (k);
        \path [-,red] (e) edge (m);
        \path [-,red] (f) edge (j);
        \path [-,red] (f) edge (k);
        \path [-,red] (f) edge (n);
        \path [-,red] (g) edge (j);
        \path [-,green] (g) edge (l);
        \path [-,red] (g) edge (m);
    \end{tikzpicture}}
    \caption{A $q$-polymatroid on $\mathbb{F}_2^3$ with cover-weights (\textcolor{green}{0},\textcolor{red}{1},2).}
    \label{fig:running_eg}
\end{figure}

\begin{example}
    Consider the cover-weighted lattice shown in Figure~\ref{fig:running_eg}. By inspection, we can conclude that the cover-weight axioms are satisfied. Therefore, by Corollary~\ref{cor:int2}, we conclude that this cover-weighted lattice corresponds to an $\mathcal{L}$-polymatroid. Since in this case the lattice is the subspace lattice of $\mathbb{F}_2^3$, it is a $q$-polymatroid.
\end{example}


\section{Cyclic Flats of an \texorpdfstring{$\mathcal{L}$}{L}-polymatroid}\label{sec:cyclicflat}
For the remainder of this paper, unless explicitly stated otherwise, $\mL$ will be assumed to be a (finite) complemented modular lattice, and $\mP=(\mathcal{L},r)$ will denote an $\mL$-polymatroid. By Theorem~\ref{thm:compmod_is_direct_prod}, we may write $\mathcal{L}$ as a direct product of subspace lattices and a Boolean lattice. 
\begin{notation}
    Let $\ell$ be a non-negative integer. We write $\mathcal{L}=\mathcal{L}_\mathbf{B}\times\mathcal{L}_\mathbf{V}$, where $\mathcal{L}_\mathbf{B}$ is a Boolean lattice and
    $\mathcal{L}_\mathbf{V}=\mathcal{L}_{V_1}\times\mathcal{L}_{V_2}\times\cdots\times\mathcal{L}_{V_\ell},$
  where $\mathcal{L}_{V_i}$ is the lattice of subspaces of a (finite) vector space $V_i$ for each $i\in[\ell]$.
\end{notation}

In this section, we introduce the notion of a cyclic flat of an $\mL$-polymatroid, which extends the analogous definitions for matroids, polymatroids, and $q$-matroids; see \cite{AlfByr22,bonin_demier,Csirm20}. Furthermore, we show that the cyclic flats of an $\mL$-polymatroid (under this definition), along with their ranks and the ranks of the atoms of the $\mL$-polymatroid, determine the rank function of the $\mL$-polymatroid.

First, we define a flat of an $\mL$-polymatroid. The definition extends naturally from the analogous concept in matroid theory.

\begin{definition}
    Let $X\in\mathcal{L}$. We say that $X$ is a \emph{flat} of $\mP$ if $r(X)<r(X\vee x)$ for every atom $x\not\leq X$.
\end{definition}

Now we introduce the notion of a cyclic element. 

\begin{definition}\label{def:cyclic_space}
We call $X\in\mathcal{L}$ \emph{cyclic} in $\mP$ if for all $H\in\mH(X)$ one of the following conditions is satisfied:
\begin{enumerate}
    \item $r(X)=r(H)$, or
    \item $0<r(X)-r(H)$ and there exists $a\in\mathcal{A}(X)\setminus\mathcal{A}(H)$ such that $r(X)-r(H)<r(a)$.
\end{enumerate}
We also refer to a cyclic element of $\mP$ as a \emph{cycle} of $\mP$.
\end{definition}

If $X\in\mathcal{L}$ is both cyclic and a flat, we call it a {\em cyclic flat}.

Definition~\ref{def:cyclic_space} generalizes the definitions of cyclic elements for polymatroids and $q$-matroids given in \cite{Csirm20} and \cite{AlfByr22}, respectively.
We remark that the second condition stated in Definition~\ref{def:cyclic_space} has the equivalent statement that if 
$0<r(X)-r(H)$ then there exists a relative complement $a$ of $H$ in $[\zero, X]$ such that $r(X)-r(H)<r(a)$.

\begin{example}
    Consider the $q$-polymatroid shown in Figure~\ref{fig:running_eg}. The set of cyclic flats of this $q$-polymatroid is
    $$\mathcal{Z}=\{0,\langle e_1, e_2\rangle, \langle e_1, e_2+e_3\rangle,\langle e_1+e_3, e_2\rangle, \mathbb{F}_2^3\}.$$
    Note that $\mathcal{Z}$ is a lattice.
\end{example}

\begin{remark}\label{rem:alt_cyc}
    The definition of a cycle given in Definition~\ref{def:cyclic_space} is not the only conceivable generalization of a cycle as defined in \cite{AlfByr22,Csirm20}. An alternative generalization of a cycle $X$ in $\mP$ is obtained by replacing the existential quantifier in Property 2 of Definition~\ref{def:cyclic_space} with the universal quantifier. We do not pursue this definition, as then the resulting set of cyclic flats is not necessarily a lattice, as we show in the next example. By contrast, as we show in Proposition~\ref{prop:cyc_flat_lattice}, with our chosen definition the set of cyclic flats does form a lattice.
\end{remark}

\begin{example}
    Consider the $q$-polymatroid shown in Figure~\ref{fig:alt_cyc_def_eg}. 
    \begin{figure}[!ht]
        \centering
        \scalebox{1.5}{
        \begin{tikzpicture}
            \node[] (0) at (0,0) {\tiny $0$};
            \node[] (1) at (0,3) {\tiny $\mathbb{F}_2^3$};
            \node[] (a) at (-3,1) {\tiny $100$};
            \node[] (b) at (-2,1) {\tiny $010$};
            \node[] (c) at (-1,1) {\tiny $110$};
            \node[] (d) at (0,1) {\tiny $111$};
            \node[] (e) at (1,1) {\tiny $011$};
            \node[] (f) at (2,1) {\tiny $001$};
            \node[] (g) at (3,1) {\tiny $101$};
            \node[] (h) at (-3,2) {\tiny $\begin{smallmatrix}100\\010\end{smallmatrix}$};
            \node[] (i) at (-2,2) {\tiny $\begin{smallmatrix}100\\011\end{smallmatrix}$};
            \node[] (j) at (-1,2) {\tiny $\begin{smallmatrix}100\\001\end{smallmatrix}$};
            \node[] (k) at (0,2) {\tiny $\begin{smallmatrix}010\\001\end{smallmatrix}$};
            \node[] (l) at (1,2) {\tiny $\begin{smallmatrix}101\\010\end{smallmatrix}$};
            \node[] (m) at (2,2) {\tiny $\begin{smallmatrix}101\\011\end{smallmatrix}$};
            \node[] (n) at (3,2) {\tiny $\begin{smallmatrix}110\\001\end{smallmatrix}$};
            \path [-,black] (0) edge (a);
            \path [-,black] (0) edge (b);
            \path [-,black] (0) edge (c);
            \path [-,black] (0) edge (d);
            \path [-,black] (0) edge (e);
            \path [-,black] (0) edge (f);
            \path [-,red] (0) edge (g);
            \path [-,green] (1) edge (h);
            \path [-,green] (1) edge (i);
            \path [-,green] (1) edge (j);
            \path [-,red] (1) edge (k);
            \path [-,green] (1) edge (l);
            \path [-,green] (1) edge (m);
            \path [-,green] (1) edge (n);
            \path [-,red] (a) edge (h);
            \path [-,red] (a) edge (i);
            \path [-,red] (a) edge (j);
            \path [-,red] (b) edge (h);
            \path [-,green] (b) edge (k);
            \path [-,red] (b) edge (l);
            \path [-,red] (c) edge (h);
            \path [-,red] (c) edge (m);
            \path [-,red] (c) edge (n);
            \path [-,red] (d) edge (i);
            \path [-,red] (d) edge (l);
            \path [-,red] (d) edge (n);
            \path [-,red] (e) edge (i);
            \path [-,green] (e) edge (k);
            \path [-,red] (e) edge (m);
            \path [-,red] (f) edge (j);
            \path [-,green] (f) edge (k);
            \path [-,red] (f) edge (n);
            \path [-,black] (g) edge (j);
            \path [-,black] (g) edge (l);
            \path [-,black] (g) edge (m);
        \end{tikzpicture}}
        \caption{A $q$-polymatroid on $\mathbb{F}_2^3$ with cover-weights (\textcolor{green}{0},\textcolor{red}{1},2).}
        \label{fig:alt_cyc_def_eg}
    \end{figure}
    If we apply the alternative definition of a cycle by replacing the existential quantifier in Property 2 of Definition~\ref{def:cyclic_space} with the universal quantifier, the set of cycles in this $q$-polymatroid becomes
    \[
        \mathcal{C} = \{0, \langle e_1, e_2 \rangle, \langle e_1, e_2+e_3 \rangle, 
        \langle e_2, e_3 \rangle, \langle e_1+e_2, e_3 \rangle \},
    \]  
    which is clearly not a lattice. If we increase the weight of every cover in Figure~\ref{fig:alt_cyc_def_eg} by 1, we obtain another $q$-polymatroid with the same set of cycles. However, in this case, each cycle also becomes a cyclic flat, meaning that $\mathcal{Z} = \mathcal{C}$, which is not a lattice.
\end{example}

For $q$-matroids, the lattice of cyclic flats, together with their ranks, determine the rank function of the $q$-matroid \cite{AlfByr22}. For polymatroids, the lattice of cyclic flats, along with their ranks and the ranks of the atoms of the ambient lattice (i.e., the singleton sets), determine the rank function of the polymatroid \cite{Csirm20}. We will show that similarly, an $\mL$-polymatroid is uniquely determined by its cyclic flats, their ranks, and the ranks of the atoms of $\mL$.

\begin{lemma}\label{lem:hyper_cycle}
    Let $C\in\mathcal{L}$ be cyclic in $\mP$. Fix $X\in[C,\one]$ and $H \in\mH(X)$. If $C\nleq H$, then $r(X)-r(H)<r(a)$ for some $a\in\mathcal{A}(X)\backslash\mathcal{A}(H)$.
\end{lemma}
\begin{proof}
    If $C\nleq H$, then $r(X)-r(H)\leq r(C)-r(C\wedge H)$ by submodularity, while $r(C)-r(C\wedge H)<r(a)$ for some $a\in\mathcal{A}(C)\backslash\mathcal{A}(C\wedge H)$ by the cyclicity of $C$. Since $a\in\mathcal{A}(X)\backslash\mathcal{A}(H)$, the result follows.
\end{proof}

\begin{lemma}\label{lem:cyc_join}
    The join of two cycles of $\mP$ is a cycle of $\mP$.
\end{lemma}
\begin{proof}
Let $C_1,C_2\in\mathcal{L}$ be cycles, and define $C=C_1\vee C_2$. Let $H\in\mH(C)$. At least one of $C_1$ and $C_2$ does not belong to $[\zero,H]$. By Lemma~\ref{lem:hyper_cycle}, either $r(H)=r(C)$ or $r(C)-r(H)<r(a)$ for some $a\in\mathcal{A}(C)\backslash\mathcal{A}(H)$. Therefore, $C$ is cyclic.
\end{proof}

\begin{definition}
    The {\em cyclic operator} of $\mP$ is defined for all $X\in\mathcal{L}$, by
    $$\textup{cyc}(X)=\bigvee\{Y\leq X:Y\textup{ is cyclic in } \mP \}.$$
\end{definition}
By Lemma~\ref{lem:cyc_join}, $\cyc(X)$ is the unique maximal cycle contained in $X$.

\begin{lemma}\label{lem:cyc(F)_is_also_flat}
    For any flat $X$ of $\mP$, $\cyc(X)$ is also a flat of $\mP$. 
\end{lemma}

\begin{proof}
    Let $X\in\mathcal{L}$ be a flat of $\mP$, and let $C=\cyc(X)$. If $C=X$, then the statement holds trivially, so suppose otherwise. Since $X$ is not cyclic, there exists $H \in \mH(X)$ such that $0<r(X)-r(H)=r(a)$ for all $a\in \mathcal{A}(X)\backslash\mathcal{A}(H)$. Therefore, by Lemma~\ref{lem:hyper_cycle}, we have that $C\leq H$ for such $H \in \mH(X)$. Furthermore, for any $a\in\mathcal{A}(\mathcal{L})\setminus\mathcal{A}(X)$, we have $r(H\vee a)-r(H)\geq r(X\vee a)-r(X)>0$ by submodularity and since $X$ is a flat. For any $a\in\mathcal{A}(X)\setminus\mathcal{A}(H)$, we have $H\vee a=X$ and so $r(H\vee a)-r(H)>0$. If $C=H$, then this provides the required flat of $\mP$. Otherwise, we observe that $C$ is also $\textup{cyc}(H)$, so we can repeat the argument with $H$ in place of $X$ and iterate until we obtain a flat $H'$ with $C=H'$.
\end{proof}

\begin{definition}\label{def:closure_function}
    For each $X\in\mathcal{L}$, the {\em closure} of $X$ in $\mP$ is given by
    $$\cl(X)=\bigvee\{x\in\mathcal{L}:r(X\vee x)=r(X)\}.$$\end{definition}

For any $X \in \mL$, the closure of $X$ in $\mP$ is clearly a flat of $\mP$. Moreover, if $X$ is already a flat of $\mP$, then $X = \cl(X)$. 

\begin{remark}
    By Lemma~\ref{lem:cyc_join}, we deduce that the set of cycles of $\mathcal{P}$ forms a lattice wherein for two cycles $X$ and $Y$, their join is $X\vee Y$ and their meet is $\textup{cyc}(X\wedge Y)$. Analogously to Lemma~\ref{lem:cyc_join}, it is easy to deduce that the meet of two flats is also a flat. Therefore, the set of flats forms a lattice wherein for two flats $X$ and $Y$, their join is $\textup{cl}(X\vee Y)$ and their meet is $X\wedge Y$.
\end{remark}

\begin{lemma}\label{lem:closure_of_cyc_space_is_cyclic}
    If $C$ is cyclic in $\mP$, then $\cl(C)$ is cyclic.
\end{lemma}

\begin{proof}
    Let $C\in\mathcal{L}$ be cyclic and let $Z=\textup{cl}(C)$. Let $H\in\mH(Z)$. If $C\leq H$, then $r(H)=r(Z)$. If $C\nleq H$, then by Lemma~\ref{lem:hyper_cycle} we have that
    there exists $a \in \mA(Z) \backslash \mA(H)$ such that $r(Z)-r(H)<r(a)$. It follows that $Z$ is cyclic.
\end{proof}

It is also clear that the closure of a cycle of $\mP$ is the minimal flat that contains that cycle. Lemma~\ref{lem:cyc(F)_is_also_flat} and Lemma~\ref{lem:closure_of_cyc_space_is_cyclic} therefore yield the following proposition.

\begin{proposition}\label{prop:cyc_flat_lattice}
    The poset of cyclic flats of an $\mL$-polymatroid is a lattice $(\mathcal{Z},\wedge_\mathcal{Z},\vee_\mathcal{Z})$ such that for all $Z_1,Z_2\in\mathcal{Z}$, we have $Z_1\wedge_\mathcal{Z}Z_2=\cyc(Z_1\wedge Z_2)$ and $Z_1\vee_\mathcal{Z}Z_2=\cl(Z_1\vee Z_2).$
\end{proposition}

\begin{example}
    The lattice of cyclic flats of the $q$-polymatroid shown in Figure~\ref{fig:alt_cyc_def_eg} is the chain 
    $\mathcal{Z}=\{0,\langle e_2,e_3\rangle,\mathbb{F}_2^3\}.$
\end{example}

\begin{definition}
    For each $X\in\mathcal{L}$, we define
    $\mathcal{B}(X)=\{\beta\subseteq\mA(X):|\beta|=\textup{h}(X),\vee\beta=X\}.$
\end{definition}

If we consider $X$ to be an element of $\mathcal{L}_\mathbf{B}\times\mathcal{L}_\mathbf{V}$, say $X=(X_{\mathbf{B}},X_{V_1},\dots,X_{V_\ell})$, then every member of $\mathcal{B}(X)$ has the form 
    $ \{ (\{b\},0_{\mL_1},\dots,0_{\mL_\ell} ): b \in X_\mathbf{B}\} \cup \{(\emptyset,\langle b_1 \rangle ,0_{\mL_2},\dots,0_{\mL_\ell} ): b_1 \in \beta_{\mathbf{V}_1}\}\cup \cdots \cup \{(\emptyset,0_{\mL_1},0_{\mL_2},\dots,0_{\mL_{\ell-1}}, \langle b_\ell \rangle): b_\ell \in \beta_{\mathbf{V}_\ell}\} $, where for each $j$, $\beta_j$ denotes a basis of the vector space $X_{V_j}$ and $\langle b_j \rangle$ is the one dimensional subspace of $V_j$ spanned by the vector $b_j$.

A key part of our approach to obtain a generalisation of \cite{Csirm20} is to use a convolution of the restriction of the rank function of $\mP$ to its lattice of cyclic flats, with a function $\mu_r$. We will show that the resulting convolution coincides with the rank function of $\mP$. 

\begin{definition}\label{not:mu_funct}
Let $f:\mA(\mathcal{L})\to\mathbb{R}_{\geq0}$ be an arbitrary function. We define the function $\mu_f:\mathcal{L}\to\mathbb{R}_{\geq0}$ such that for all  $A \in \mL$ we have
\[
\mu_f(A)=\textup{min}\left\{\sum_{x\in \beta}f(x):\beta\in \mathcal{B}(A)\right\}.
\]
\end{definition}

The function $\mu_f$ of Definition~\ref{not:mu_funct} is a generalisation of the analogous function found in \cite{Csirm20}, where $\mL$ is the Boolean lattice, and thus $\mu_f(A)=\sum_{a\in\mA(A)}f(a)$. In the case that $\mathcal{L}$ is a Boolean lattice, it is easy to see that $\mu_f$ is a valuation (or modular function) (see \cite[Chapter X]{birkhoff_lattice}) on $\mL$. That is, for all $A,B\in\mathcal{L}$ we have 
$$\mu_f(A)+\mu_f(B)=\mu_f(A\wedge B)+\mu_f(A\vee B).$$


\begin{remark}
    For an arbitrary $\mL$-polymatroid $\mP=(\mL,r)$, the function $\mu_r$ is not necessarily increasing on $\mathcal{L}$. Moreover, if $\mu_r$ is not increasing, then it is not a valuation on $\mL$. Indeed, if $\mu_r$ were a valuation, then for any $A,B \in \mL$ satisfying $A \wedge B = \zero$, we would have  
    \[
    \mu_r(A\dot \vee B) = \mu_r(A\vee B) + \mu_r(A \wedge B) = \mu_r(A) + \mu_r(B) \geq \mu_r(A).
    \]  
    We illustrate this remark in the following example.
\end{remark}

\begin{example}
    Consider the $q$-polymatroid of rank 2 shown in Figure \ref{fig:munotinc}, whose lattice of cyclic flats is $\mathcal{Z}=\{0,\mathbb{F}_2^3\}$. It is straightforward to check that $\mu_r(\langle e_1,e_2 \rangle)=4$, while 
    $\mu_r(\one)=3$.
    If $\mu_r$ were a valuation, then it would be non-decreasing, so $\mu_r$ is not a valuation. 
    \begin{figure}[ht!]
    \centering
    \scalebox{1.5}{
    \begin{tikzpicture}
        \node[] (0) at (0,0) {\tiny $0$};
        \node[] (1) at (0,3) {\tiny $\mathbb{F}_2^3$};
        \node[] (a) at (-3,1) {\tiny $100$};
        \node[] (b) at (-2,1) {\tiny $010$};
        \node[] (c) at (-1,1) {\tiny $110$};
        \node[] (d) at (0,1) {\tiny $111$};
        \node[] (e) at (1,1) {\tiny $011$};
        \node[] (f) at (2,1) {\tiny $001$};
        \node[] (g) at (3,1) {\tiny $101$};
        \node[] (h) at (-3,2) {\tiny $\begin{smallmatrix}100\\010\end{smallmatrix}$};
        \node[] (i) at (-2,2) {\tiny $\begin{smallmatrix}100\\011\end{smallmatrix}$};
        \node[] (j) at (-1,2) {\tiny $\begin{smallmatrix}100\\001\end{smallmatrix}$};
        \node[] (k) at (0,2) {\tiny $\begin{smallmatrix}010\\001\end{smallmatrix}$};
        \node[] (l) at (1,2) {\tiny $\begin{smallmatrix}101\\010\end{smallmatrix}$};
        \node[] (m) at (2,2) {\tiny $\begin{smallmatrix}101\\011\end{smallmatrix}$};
        \node[] (n) at (3,2) {\tiny $\begin{smallmatrix}110\\001\end{smallmatrix}$};
        \path [-,black] (0) edge (a);
        \path [-,black] (0) edge (b);
        \path [-,black] (0) edge (c);
        \path [-,red] (0) edge (d);
        \path [-,red] (0) edge (e);
        \path [-,red] (0) edge (f);
        \path [-,red] (0) edge (g);
        \path [-,green] (1) edge (h);
        \path [-,green] (1) edge (i);
        \path [-,green] (1) edge (j);
        \path [-,green] (1) edge (k);
        \path [-,green] (1) edge (l);
        \path [-,green] (1) edge (m);
        \path [-,green] (1) edge (n);
        \path [-,green] (a) edge (h);
        \path [-,green] (a) edge (i);
        \path [-,green] (a) edge (j);
        \path [-,green] (b) edge (h);
        \path [-,green] (b) edge (k);
        \path [-,green] (b) edge (l);
        \path [-,green] (c) edge (h);
        \path [-,green] (c) edge (m);
        \path [-,green] (c) edge (n);
        \path [-,red] (d) edge (i);
        \path [-,red] (d) edge (l);
        \path [-,red] (d) edge (n);
        \path [-,red] (e) edge (i);
        \path [-,red] (e) edge (k);
        \path [-,red] (e) edge (m);
        \path [-,red] (f) edge (j);
        \path [-,red] (f) edge (k);
        \path [-,red] (f) edge (n);
        \path [-,red] (g) edge (j);
        \path [-,red] (g) edge (l);
        \path [-,red] (g) edge (m);
    \end{tikzpicture}}
    \caption{A $q$-polymatroid with cover weights (\textcolor{green}{0},\textcolor{red}{1},2).}
    \label{fig:munotinc}
    \end{figure} 
\end{example}

For the remainder of the paper, we fix $f$ to be a non-negative real-valued function defined on the atoms of $\mathcal{L}$.

Lemma~\ref{lem:mu_submod} gives an instance of submodularity that we will use frequently.

\begin{lemma}\label{lem:mu_submod}
    Let $A,B\in\mathcal{L}$ such that $A\wedge B=\zero$. Then we have $\mu_f(A\dot\vee B)\leq \mu_f(A)+\mu_f(B)$.
\end{lemma}

\begin{proof}
    Let $\beta_1\in\mathcal{B}(A)$ and $\beta_2\in\mathcal{B}(B)$ such that 
    $\displaystyle \mu_f(A) = \sum_{a\in\beta_1} f(a)$ and $\displaystyle \mu_f(B) = \sum_{a\in\beta_2} f(a)$.
    Since $A\wedge B=\zero$, recalling that $\mathcal{L}=\mathcal{L}_\mathbf{B}\times\mathcal{L}_\mathbf{V}$, we have $\beta_1\dot\cup\beta_2 \in\mathcal{B}(A\dot\vee B)$. Thus, we obtain
    \[
    \mu_f(A\dot\vee B) \leq \sum_{a\in\beta_1\dot\cup\beta_2} f(a) 
    = \sum_{a\in\beta_1} f(a) + \sum_{a\in\beta_2} f(a) 
    = \mu_f(A) + \mu_f(B). \qedhere
    \]
\end{proof}

The following definition will be useful in many later statements.

\begin{definition}\label{def:layers}
    Let $X\in\mathcal{L}$, and let $\zero=H_m\lessdot\dots\lessdot H_0 =X$ be a maximal chain from $\zero$ to $X$. For $k \in [m]$, we define
    \[
    L_k=\{a\in\mA(X):a\leq H_{k-1},a\nleq H_k\},
    \]
    which we call the $k$-th {\em layer} (or simply a layer) of the chain. The list of sets $(L_1,\dots,L_m)$ is the \emph{associated layering} of this chain. 
\end{definition}

 By Theorem~\ref{thm:compmod_is_direct_prod}, and elementary results on linear algebra, we have the following statement on properties of the associated layering of a maximal chain.
\begin{lemma}\label{lem:indep_layers_2_properties}
    Let $X\in\mathcal{L}$ and let $\zero\lessdot H_m\lessdot\cdots\lessdot H_0=X$ be a maximal chain with associated layering $(L_1,\dots,L_m)$. The following properties hold:
    \begin{enumerate}
        \item Let $\{a_1,\dots,a_k\}\subseteq\mathcal{A}(\mathcal{L})$ such that $a_i\in L_{n_i}$ and $n_k<n_{k-1}<\cdots<n_1$. Then $\{a_1,\dots,a_k\}$ is independent.
        \item For any $Y\leq X$, the set $\mA(Y)$ intersects exactly $\textup{h}(Y)$ layers $L_i$ nontrivially.
        \item For any $Y \leq X$, there exists $S \subseteq[m], |S|=\h(Y)$ and $a_j \in L_j$ such that $\bigvee_{j \in S} a_j = Y$.
    \end{enumerate}
\end{lemma}

\begin{lemma}\label{lem:max_chain_ranks}
    Let $X\in\mathcal{L}$. There exists a chain $\cyc(X)=H_k\lessdot H_{k-1}\lessdot\cdots\lessdot H_0=X$ such that for all $i\in[k]$ and all $a_i\in \mathcal{A}(H_{i-1})\setminus\mathcal{A}(H_i)$, the following properties hold. 
    \begin{enumerate}
        \item $r(H_i)-r(H_{i-1})=r(a_{i})>0$.
        \item $r(X)-r(\cyc(X))=\sum_{i=1}^kr(a_i)$.
    \end{enumerate}
\end{lemma}
\begin{proof}
    The case for which $X=\cyc(X)$ clearly holds. Suppose that $X$ is not cyclic, which means that $k=\len([\cyc(X),X])>0$. By definition, there exists $H_1\in\mathcal{H}(X)$ such that $0<r(X)-r(H_1)=r(a_1)$ for all $a_1\in\mathcal{A}(X)\setminus\mathcal{A}(H_1)$. By Lemma~\ref{lem:hyper_cycle}, we have $\cyc(X)\leq H_1$. By definition, we deduce that $\cyc(H_1)=\cyc(X)$. We therefore may repeat this process until we have generated a maximal chain in $[\cyc(X),X]$. This gives us Property 1, from which Property 2 immediately follows.
\end{proof}

The following lemma generalizes \cite[Lemma 1]{Csirm20}.

\begin{lemma}\label{lem:cyc_rank_identity}
    Let $X \in \mL$ and let $A\in\mathcal{L}$ be such that $\textup{cyc}(X)\leq A\leq X$. Then, for any $\textup{cyc}(X)^c\in \cC(\textup{cyc}(X))$, we have
    \(
    r(A) = r(\textup{cyc}(X))+\mu_r(\textup{cyc}(X)^c\wedge A).
    \)
\end{lemma}

\begin{proof}
    It is clear by definition that $\cyc(A)=\cyc(X)$. The case that $A=\textup{cyc}(X)$ is trivial. Suppose that $A$ is not cyclic, which means that $k=\len([\cyc(X),A])>0$. Let $(L_1,\dots,L_k)$ be the associated layering of the chain $H_k \lessdot \cdots \lessdot H_0$. By Lemma~\ref{lem:max_chain_ranks}, it follows that for any $a_i\in L_i$ for each $i\in[k]$, we have
    \begin{align}
        r(A)-r(\cyc(X)) &= \sum_{i=1}^{k}r(a_i)\label{eq:cyc_layers_lemma}.
    \end{align}
    Since $\cyc(X)\leq A$, Proposition~\ref{prop:compdecomp} implies that $\mathbf{C}(\cyc(X))=\mathbf{C}(\cyc(X);A)$. By the modularity of $\mathcal{L}$, we deduce that $\textup{h}(\cyc(X)^c\wedge A)=k$ for any $\cyc(X)^c\in\mathbf{C}(\cyc(X))$. By Lemma~\ref{lem:indep_layers_2_properties}, we thus have that $\mathcal{A}(\cyc(X)^c\wedge A)\cap L_i\neq\varnothing$ for all $i\in[k]$. We hence can choose $b_i\in\mathcal{A}(\cyc(X)^c\wedge A)\cap L_i$ for each $i\in[k]$. By Lemma~\ref{lem:indep_layers_2_properties}, we have $\textup{h}(\bigvee_{i=1}^kb_i)=k$, from which we conclude that $\{b_1,\dots,b_k\}\in\mathcal{B}(\cyc(X)^c\wedge A)$. Recall that (\ref{eq:cyc_layers_lemma}) holds for arbitrary $a_i\in L_i$. By the definition of $\mu_r$, axiom \ref{r3}, and (\ref{eq:cyc_layers_lemma}), we hence obtain
    \[
    \mu_r(\textup{cyc}(X)^c\wedge A)\leq\displaystyle\sum_{i=1}^{k}r(b_i)=r(A)-r(\textup{cyc}(X))\leq r(\textup{cyc}(X)^c\wedge A)\leq\mu_r(\textup{cyc}(X)^c\wedge A).
    \]
    Since our choice of $\cyc(X)^c\in\mathbf{C}(\cyc(X))$ was arbitrary, the result follows.
\end{proof}

We are now ready to state the main result of this section, which says that the rank function of an $\mL$-polymatroid is determined by the lattice of cyclic flats $\mZ$, the ranks of the elements of $\mZ$, and the ranks of the atoms of $\mathcal{L}$.

\begin{theorem}\label{thm:convol_identity}
    Let $\mathcal{Z}$ be the lattice of cyclic flats of $\mP=(\mathcal{L},r)$. For any $X\in\mathcal{L}$, we have 
    \[
       r(X)=\textup{min}\{r(Z)+\mu_r(Z^c\wedge X):Z\in\mathcal{Z}, Z^c\in \mathbf{C}(Z;X)\}.
    \]   
    Moreover, if $Z=\textup{cyc}(\textup{cl}(X))$, then
    \(r(X)=r(Z)+\mu_r(X\wedge Z^c)\) for any $Z^c\in\mathbf{C}(Z;X)$.
\end{theorem}

\begin{proof}
    By \ref{r3}, for any $Z\in\mathcal{L}$ and $Z^c\in \mathbf{C}(Z;X)$ we have     
    $$r(X) = r((Z \wedge X) \dot\vee (Z^c \wedge X))\leq r(Z)+r(Z^c\wedge X)\leq r(Z)+\mu_r(Z^c\wedge X).$$
    In particular, 
    $$r(X) \leq \min \{ r(Z)+\mu_r(Z^c\wedge X):Z\in\mathcal{Z}, Z^c\in \mathbf{C}(Z;X)\}.$$
    
    We now show that this minimum is attained for $Z = \cyc(\cl(X))$.
    Let $F=\textup{cl}(X)$ and let $Z=\textup{cyc}(F)$. 
    Since $Z\leq Z\vee X\leq F$, we use the definition of $\textup{cl}(X)$, as well as \ref{r2}, to deduce 
    $$r(X)\leq r(Z \vee X)\leq r(F)=r(X).$$ 
    Furthermore, by Lemma~\ref{lem:cyc_rank_identity}, for any $Z^c\in\mathbf{C}(Z)$ we obtain 
    \[
      r(X)=r(Z\vee X)=r(Z)+\mu_r(Z^c\wedge(Z\vee X)).
    \]
    
    We now show that $Z^c\wedge X= Z^c\wedge(Z\vee X)$ whenever $Z^c\in \mathbf{C}(Z;X)$.
    Clearly, we have $Z^c\wedge X\leq Z^c\wedge(Z\vee X)$. To show equality, we compute their heights. Since $Z\leq Z\vee X$, the modularity of $\mathcal{L}$ gives us that
    $$\textup{h}(Z^c\wedge(Z\vee X))=\textup{h}(Z\vee X)-\textup{h}(Z)=\textup{h}(X)+\textup{h}(Z)-\textup{h}(X\wedge Z)-\textup{h}(Z)=\textup{h}(X)-\textup{h}(X\wedge Z).$$
    Since $Z^c \in \cC(Z;X)$ then $\textup{h}(X)=\textup{h}(X\wedge Z)+\textup{h}(X\wedge Z^c)$ and hence we get $Z^c\wedge X= Z^c\wedge(Z\vee X)$. The result follows.
\end{proof}

We include the following simple example to highlight the use of decomposing complements and the function $\mu_r$.

\begin{example}
    Consider the $q$-polymatroid shown in Figure~\ref{fig:alt_cyc_def_eg}, whose lattice of cyclic flats is $\mathcal{Z}=\{0,\langle e_2,e_3\rangle,\mathbb{F}_2^3\}$. We consider the space $\langle e_1,e_2\rangle\leq\mathbb{F}_2^3$. In the following, we compute $r(Z)+\mu_r(Z^c\wedge\langle e_1,e_2\rangle)$ for each $Z\in\mathcal{Z}$ and each $Z^c\in\mathbf{C}(Z;\langle e_1,e_2\rangle)$. Observe that 
    $$\mathbf{C}(\mathbb{F}_2^3;\langle e_1,e_2\rangle)=\{0\},\quad\mathbf{C}(0;\langle e_1,e_2\rangle)=\{\mathbb{F}_2^3\},\textup{ and}$$
    $$\mathbf{C}(\langle e_2,e_3\rangle;\langle e_1,e_2\rangle)=\{\langle e_1\rangle,\langle e_1+e_2\rangle\}.$$
    We now compute
    \begin{align*}
        r(\mathbb{F}_2^3)+\mu_r(0\wedge\langle e_1,e_2\rangle)&=3+0=3,\\
        r(\langle e_2,e_3\rangle)+\mu_r(\langle e_1\rangle\wedge \langle e_1,e_2\rangle)&=2+2=4,\\
        r(\langle e_2,e_3\rangle)+\mu_r(\langle e_1+e_2\rangle\wedge \langle e_1,e_2\rangle)&=2+2=4,\\
        r(0)+\mu_r(\mathbb{F}_2^3\wedge\langle e_1,e_2\rangle)&=0+4=4,
    \end{align*}
    which agrees with $r(\langle e_1,e_2\rangle)=3$. We highlight the importance of taking a decomposing complement by pointing out that if we take $\langle e_2,e_3\rangle^c=\langle e_1+e_3\rangle\in\mathbf{C}(\langle e_2,e_3\rangle)\setminus\mathbf{C}(\langle e_2,e_3\rangle;\langle e_1,e_2\rangle)$, then we have
    $$r(\langle e_2,e_3\rangle)+\mu_r(\langle e_1+e_3\rangle\wedge \langle e_1,e_2\rangle)=r(\langle e_2,e_3\rangle)+\mu_r(0)=2+0=2<r(\langle e_1,e_2\rangle).$$
\end{example}

\begin{remark}\label{rem:q-mat_mu}
    If $\mP$ is a $q$-matroid, then for any $X\in\mathcal{L}$, we have  
    \[
        \mu_r(X) = \textup{h}(X) - \textup{h}(\textup{cl}(\zero) \wedge X) = \dim(X) - \dim(\textup{cl}(\zero) \cap X).
    \]
    For example, if $Z\in\mathcal{L}$ is a cyclic flat, then any $Z^c\in\mathbf{C}(Z)$ contains no loops. Therefore, if $Z=\textup{cyc}(\textup{cl}(X))$, then for any complement $Z^c$ of $Z$ in $\mL$, we have  
    \[
        r(X) = r(Z) + \mu_r(X \cap Z^c) = r(Z) + \dim(X \cap Z^c) = r(Z) + \dim(X / (X \cap Z)).
    \]
    See \cite[Lemma 2.25]{AlfByr22}, for example.
\end{remark}

We now recall some fundamental results from matroid theory that we will use to analyse the structure of the function $\mu_f$. The following theorem recalled from \cite{welsh} is stated in terms of geometric lattices; since complemented modular lattices form a subclass of these, it applies in our setting.
\begin{theorem}[{\cite[Chapter 3.3 Theorem 1]{welsh}}]\label{thm:mat_flat_lattice}
    A finite lattice is isomorphic to the lattice of flats of a matroid if and only if it is geometric.
\end{theorem}
The following theorem of Rado is recalled from \cite{welsh}. For its statement, we let $E$ be a finite set, we let $\hat{f}:E\to\mathbb{R}_{\geq0}$ and we extend $\hat{f}$ to a function on $2^E$ by $\hat{f}(X)=\sum_{x\in X}\hat{f}(x)$ for each $X\subseteq E$. For any $\mathcal{I}\subseteq 2^E$, the \emph{optimisation problem} $(\mathcal{I},\hat{f})$ is to find $I \in \mathcal{I}$ of maximal cardinality such that 
$\hat{f}(I)$ is minimal. 
\begin{theorem}[{\cite[Chapter 19.1 Theorem 1]{welsh}}]\label{thm:mat_greedy}
    If $\mathcal{I}$ is the collection of independent sets of a matroid on $E$, then the greedy algorithm works for the optimisation problem $(\mathcal{I},\hat{f})$.
\end{theorem}
\begin{definition}
    Let $M=(E,\mathcal{I})$ be a matroid such that its lattice of flats $\mathcal{F}$ is isomorphic to $\mathcal{L}$. Let $\Phi:\mathcal{F}\to\mathcal{L}$ be a lattice isomorphism. Let $\hat{f}:E\to\mathbb{R}_{\geq0}$ be defined by $\hat{f}(x)=f(\Phi(\cl(\{x\})))$ for all $x \in E$ and extend $\hat{f}$ to $2^E$ by $\hat{f}(X)=\sum_{x\in X}\hat{f}(x)$ for each $X\subseteq E$.
    Let $S \in 2^E$ and let $X \subseteq S$. 
    We say that $X$ is $\hat{f}$\emph{-minimal} in $S$ if
    $\hat{f}(X) = \min\{\hat{f}(Y) : Y \subseteq S\}$.
\end{definition}
For the remainder, we fix $M=(E,\mathcal{I})$ to be an arbitrary loopless matroid with lattice of flats $\mathcal{F}$ isomorphic to $\mL$ under the lattice isomorphism $\Phi$.

\begin{algorithm}
\caption{Greedy algorithm on the atoms of $A$ for any $f:\mA(\mL) \to \RR_{\geq 0}$ and $A \in \mL$}\label{alg:mu}
\begin{algorithmic}
\State $V \gets \zero$
\State $S \gets 0$
\While{$\mA(A) \setminus \mA(V) \neq \varnothing$}
    \State Choose $e \in \mA(A) \setminus \mA(V)$ such that $f(e)$ is minimal
    \State $S \gets S + f(e)$
    \State $V \gets V \vee e$
\EndWhile
\end{algorithmic}
\end{algorithm}

Algorithm~\ref{alg:mu} is a greedy algorithm on the atoms of $\mathcal{L}$, which computes $\mu_f(A)$ for any $A\in\mathcal{L}$.

\begin{lemma}\label{lem:greedy_alg}
    Let $A\in\mathcal{L}$ and let $f:\mA(\mL)\to \RR_{\geq 0}$. Then Algorithm~\ref{alg:mu} computes $\mu_f(A)$. 
\end{lemma}
\begin{proof}
    Computing $\mu_f(A)$ is equivalent to solving the optimisation problem $(\mathcal{I}|(\Phi^{-1}(A)),\hat{f})$. Therefore, by Theorem~\ref{thm:mat_greedy}, the result follows.
\end{proof}

\begin{lemma}\label{lem:mu_of_join_of_layers}
    Let $X\in\mathcal{L}$, and let $\zero=H_m\lessdot\dots\lessdot H_0=X$ be a maximal chain with associated layering $(L_1,\dots,L_m)$. 
    Let $f:\mA(\mL)\to \RR_{\geq 0}$.
    Suppose that for any $j,k\in[m]$ with $j\leq k$ we have that
    \begin{enumerate}
        \item if $a,a'\in L_j$, then $f(a)=f(a')$, and
        \item if $a\in L_j$ and $a'\in L_k$, then $f(a)\geq f(a')$.
    \end{enumerate}
    Let $a_j\in L_j$ for each $j\in[m]$, and let $S\subseteq[m]$. Then we have 
        $\displaystyle \mu_f\left(\bigvee_{i\in S}a_i\right)=\sum_{i\in S}f(a_i).$
\end{lemma}
\begin{proof}
    By Properties 1 and 2 (of the lemma statement), there exists a partition of $S$ of the form $S_1\dot\cup\cdots\dot\cup S_\ell$ such that $f(a_i)=f(a_j)$ if $i,j\in S_k$ and $f(a_i)>f(a_j)$ if $i\in S_k$ and $j\in S_{k+1}$ for all possible $i,j,k$. Also, for any $i\in S$ such that $f(a_i)<f(a_{i-1})$ (i.e., there is a strict inequality), we obtain
    \begin{equation}\label{eq:mu_sum_layers}
        A\wedge H_i=\bigvee\{a\in\mathcal{A}(A):a\in L_j\textup{ for }j\geq i\}=\bigvee\{a\in\mathcal{A}(A):f(a)\leq f(a_i)\}.
    \end{equation}

    By Lemma~\ref{lem:indep_layers_2_properties}, $A=\bigvee_{i\in S}a_i$ intersects precisely $|S|$ layers, so for each $a\in\mathcal{A}(A)$ there exists $i\in S$ such that $a\in L_i$ (and so $f(a)=f(a_i)$, by Propoerty 1). We now apply Algorithm~\ref{alg:mu} to the atoms of $A$. The first $|S_1|$ atoms of $A$ are chosen from layers $L_i$ such that $i\in S_1$. By (\ref{eq:mu_sum_layers}), we deduce that the next atom must be chosen from $L_i$ such that $i\notin S_1$. We deduce that the next $|S_2|$ atoms of $A$ are chosen from layers $L_i$ such that $i\in S_2$. This process thus repeats until we have chosen $|S|$ independent atoms. This gives a sum of evaluations of $f$, which by Lemma~\ref{lem:greedy_alg}, equals $\mu_f(A)$. By Property 1 and the construction of the $S_i$, we deduce that this sum equals $\sum_{i\in S}f(a_i)$.
\end{proof}

We now define a class of functions
, which we call weakly decomposable. Such functions are fundamental to the main result of this paper: in Section~\ref{sec:main}, we require $\mu_f$ to be weakly decomposable on certain intervals as part of the axiomatisation of cyclic flats.

\begin{definition}
    Let $g:\mathcal{L}\to\mathbb{R}_{\geq 0}$ be a function, and let $X\in\mathcal{L}$. We say that $g$ is \emph{weakly decomposable} on $[\zero,X]$ if for all $Y\in[\zero,X]$, there exists $Y^c\in\mathbf{C}(Y)$ such that 
    \[
        g(X)=g(Y)+g(X\wedge Y^c).
    \]
\end{definition}

\begin{remark}
    If $\mathcal{L}$ is a Boolean lattice, then $\mu_f$ is a valuation, which implies that it is weakly decomposable.
\end{remark}

The following is an explicit example of a weakly decomposable function on the subspace lattice of $\mathbb{F}_2^4$.

\begin{example}
    Suppose that $\mathcal{L}$ is the subspace lattice of $\mathbb{F}_q^4$. Let $e_1,e_2,e_3,e_4$ be the standard basis vectors. Define the following vector spaces
    $$H_j=\bigoplus_{i=1}^j\langle e_i\rangle\quad\textup{for }j=1,2,3.$$
    Consider the function $g:\mathcal{L}\to\mathbb{Z}$ defined by
    $g(X)=4\cdot\dim(X)-\sum_{i=1}^3\dim(X\cap H_i).$
    By inspection, we can conclude that $g$ is not a valuation, but that $g$ is weakly decomposable on $[0,\mathbb{F}_q^4]$. For instance, consider
    $g(\mathbb{F}_q^4)=4\cdot 4-3-2-1=10$ and $g(\langle e_1+e_2,e_3+e_4\rangle)=4\cdot 2-1-1=6.$
    The vector space $\langle e_1,e_3\rangle$ is a complement of $\langle e_1+e_2,e_3+e_4\rangle$ such that 
    $g(\langle e_1,e_3\rangle)=4\cdot 2-2-1-1=4,$
    which gives $g(\mathbb{F}_q^4)=g(\langle e_1+e_2,e_3+e_4\rangle)+g(\langle e_1,e_3\rangle).$
    However, $\langle e_2,e_4\rangle$ is another complement of $\langle e_1+e_2,e_3+e_4\rangle$, but
    $g(\langle e_2,e_4\rangle)=4\cdot 2-1-1=6,$
    from which we obtain
    $g(\mathbb{F}_q^4)=10<12=g(\langle e_1+e_2,e_3+e_4\rangle)+g(\langle e_2,e_4\rangle).$
\end{example}

The following result on weakly decomposable functions will be used in the proof of Proposition~\ref{cor:mu_is_increasing}, for which we omit the simple proof.

\begin{lemma}\label{lem:qmBA}
    Let $g:\mathcal{L}\to\mathbb{R}_{\geq 0}$ be a function, and let $X \in \mathcal{L}$.  
    If $g$ is weakly decomposable on $[\zero,X]$, then $g(X) \geq g(Y)$ for all $Y \in [\zero,X]$.  
\end{lemma}

\begin{lemma}\label{lem:mu_is_q-mod_on_cf}
    Let $X \in \mathcal{L}$ and let $C = \textup{cyc}(X)$. 
    Then the function $\mu_r$ is weakly decomposable on $[\zero, X\wedge C^c]$ for any $C^c \in \mathbf{C}(C)$.
\end{lemma}

\begin{proof}
    If $X=\cyc(X)$, then the statement holds trivially, so suppose that $X$ is not cyclic and hence $m=\len([\cyc(X),X])>0$. 
    Let $C^c\in\mathbf{C}(C)$. By Lemma~\ref{lem:cyc_rank_identity}, we have that
        $r(X)-r(\cyc(X))=\mu_r(X\wedge C^c).$
    By Lemma~\ref{lem:max_chain_ranks}, there exists a chain $C=H_m\lessdot\cdots\lessdot H_0=X$ and an associated layering 
    $(L_1,\dots,L_m)$ such that for all $i\in[m]$ and all $a_i\in L_i$, we have 
    $\sum_{i=1}^mr(a_i) = \mu_r(X\wedge C^c)$ and $r(H_{i-1})-r(H_i)=r(a_i)>0$.
    In particular, $r$ is constant on each $L_i$. 
    Let $A\leq X\wedge C^c$. By Lemma~\ref{lem:indep_layers_2_properties}, there exists $S\subseteq[m]$ and $b_i\in L_i$ for each $i\in S$ such that $A=\bigvee_{i\in S}b_i$ and $\{b_i:i\in S\}\in\mathcal{B}(A)$. Moreover, we have that $B=\bigvee_{i\in[m]\setminus S}a_i$ is a relative complement of $A$ in $[\zero,X\wedge C^c]$ and $\{a_i:i\in [m]\setminus S\}\in\mathcal{B}(B)$. By Lemma \ref{cor:extend_to_complement}, $B$ extends to a complement $A$. Since $r$ is constant on each layer, we have $\mu_r(A) \leq \sum_{i\in S}r(b_i)=\sum_{i\in S}r(a_i)$ and 
    $\mu_r(B)\leq \sum_{i\in [m] \backslash S}r(b_i)=\sum_{i\in i\in [m] \backslash S}r(a_i)$. Therefore, by Lemma~\ref{lem:mu_submod}, we have
    $$\mu_r(X\wedge C^c)\leq\mu_r(A)+\mu_r(B)\leq\sum_{i\in S}r(a_i)+\sum_{i\in[m]\setminus S}r(a_i)=\sum_{i\in[m]}r(a_i)=\mu_r(X\wedge C^c).$$
    The result follows.
    \end{proof}

The remainder of this section focuses on establishing results about $\mu_f$, which we apply in Section~\ref{sec:main}.

\begin{lemma}\label{lem:weak_val_mat_bases}
    Let $X\in\mathcal{L}$. The function $\mu_f$ is weakly decomposable on $[\zero,X]$ if and only if every flat of $M|\Phi^{-1}(X)$ contains a basis that is contained in an $\hat{f}$-minimal basis of $M|\Phi^{-1}(X)$.
\end{lemma}
\begin{proof}
    For any $A \in [\zero,X]$, $\mu_f(A)=\min\{\hat{f}(B):B\textup{ is a basis of }
    M|\phi^{-1}(A)\}$.
    We have that $\mu_f$ is weakly decomposable on $[\zero,X]$ if for every $Y \in[\zero,X]$, there exists $Z \in [\zero, X]$ such that $X=Y\dot\vee Z$ and
    $\mu_f(X)=\hat{f}(B_X) = \hat{f}(B_Y)+\hat{f}(B_Z)$ for some 
    $\hat{f}$-minimal bases $B_X$, $B_Y$, and $B_Z$ of $M|(\Phi^{-1}(X))$, $M|(\Phi^{-1}(Y))$, and $M|(\Phi^{-1}(Z))$, respectively, 
in which case we have that $B_Y\cup B_Z$ is an $\hat{f}$-minimal basis of $M|(\Phi^{-1}(X))$. 
\end{proof}

\begin{lemma}\label{lem:q-mod_implies_mu-support}
    Let \(X\in\mathcal{L}\), and suppose that \(\mu_f\) is weakly decomposable on \([\zero,X]\). If \(\textup{h}(X) = m\), then there exists a maximal chain $\zero = H_m \lessdot H_{m-1} \lessdot \dots \lessdot H_0 = X$
    and an associated layering \(L_1, \dots, L_m\) such that:  
    \begin{enumerate}
        \item For any \(k\in[m]\) and all \(a, a' \in L_k\), we have \(f(a) = f(a')\).  
        \item For any \(k \in [m-1]\), if \(a \in L_k\) and \(a' \in L_{k+1}\), then \(f(a) \geq f(a')\).  
    \end{enumerate}
\end{lemma}

\begin{proof}
    Let $M'=M|(\Phi^{-1}(X))$. By Lemma~\ref{lem:weak_val_mat_bases}, that $\mu_f$ is weakly decomposable on $[\zero, X]$ is equivalent to the statement that for any $F\in\mathcal{F}$ contained in $\Phi^{-1}(X)$, an $\hat{f}$-minimal basis of $M'|F$ is contained in an $\hat{f}$-minimal basis of $M'$. Using the greedy algorithm, choose a basis $\{b_1,b_2,\dots,b_m\}$ of $M'$ such that $\hat{f}(b_i)\leq\hat{f}(b_{j})$ for $i\leq j$. Generate a chain of flats $\varnothing=F_m\subset F_{m-1}\subset\cdots\subset F_0=\Phi^{-1}(X)$ with $F_i=\cl(\{b_1,\dots,b_{m-i}\}), 0\leq i\leq m-1$. For any $i\in[m]$, any basis $B$ of $M'|F_i$, and any $a\in X$, the set $B\cup\{a\}$ is in $\mathcal{I}$ if and only if $a\notin F_i$. By our use of the greedy algorithm, we thus deduce that $\hat{f}(a')\leq\hat{f}(a)$ for any $a'\in F_i$ and $a\notin F_i$. By setting $H_k=\Phi(F_k)$ for each $k\in[m]$, Property 2 of the lemma statement follows.

    By an elementary result of matroid theory, any basis of $M'$ that contain a basis $J$ of $M'|F_i$ is of the form $J\cup K$ for some basis $K$ of $M'/F_i$. Therefore, for any given basis $J$ of $M'|F_i$, an $\hat{f}$-minimal basis of $M'$ containing $J$ must be of the form $J\cup K$ where $K$ is an $\hat{f}$-minimal basis of $M'/F_i$.
    
    We now show Property 1. Suppose, towards a contradiction, that there exist $a,a'\in F_i\setminus F_{i+1}$ such that $\hat{f}(a)<\hat{f}(a')$. For any basis $I$ of $M'|F_{i+1}$, we have that $I\cup \{a\}$ and $I\cup \{a'\}$ are both independent. By Property 2 of this lemma, we deduce that $a'$ is not contained in an $\hat{f}$-optimal basis $J$ of $M'|F_i$. Therefore, $a'$ is not contained in an $\hat{f}$-minimal basis of $M'$, which we observed to have the form $J\cup K$ for some $\hat{f}$-minimal basis $K$ of $M'/F_i$. By the definition of $\hat{f}$, it follows that $\cl(\{a'\})$ does not contain a basis that is contained in an $\hat{f}$-optimal basis of $M'$, which yields a contradiction by Lemma~\ref{lem:weak_val_mat_bases}. The result follows.
\end{proof}

\begin{corollary}\label{cor:mu_q-mod_gives_additive}
    Let $X\in\mathcal{L}$, and let $\mu_f$ be weakly decomposable on $[\zero,X]$. 
    Then there exists a maximal chain $\zero=H_m\lessdot\dots\lessdot H_0=X$ with associated layering $L_1,\dots,L_m$ such that for any $A \in [\zero ,X]$, for any
    $a_j\in L_j$ and subset $S \subseteq [m]$ such that $A=\bigvee_{j \in S} a_j$ we have $\mu_f(A)=\sum_{i\in S}f(a_i)$.
\end{corollary}
\begin{proof}
    We apply Algorithm~\ref{alg:mu} to construct a maximal chain
    $\zero=H_m\lessdot\dots\lessdot H_0=X$ and its associated layering $L_1,\dots,L_m$. 
    By Lemma~\ref{lem:q-mod_implies_mu-support}, Properties 1 and 2 of Lemma~\ref{lem:q-mod_implies_mu-support} hold. Therefore, the conditions of Lemma~\ref{lem:mu_of_join_of_layers} are satisfied on $[\zero, X]$. By Lemma~\ref{lem:indep_layers_2_properties}, there exists a subset $S \subseteq [m]$ such that $A=\bigvee_{j \in S} a_j$, where $|S|=\h(A)$ and $a_j \in L_j$ for each $j$. It is clear that the conditions of Lemma~\ref{lem:mu_of_join_of_layers} are thus also satisfied on $[\zero, A]$, which means that $\mu_f(A)=\sum_{j\in S}f(a_j)$. By Property 1 of Lemma~\ref{lem:q-mod_implies_mu-support}, $f$ is constant on each layer
    and so $\mu_f(A)$ is independent of our choice of $a_j\in L_j$ satisfying $A=\bigvee_{j\in S}a_j$.
\end{proof}

The following result highlights an important feature of a weakly decomposable function.

\begin{proposition}\label{cor:mu_is_increasing}
    Let $X \in \mL$. 
    If $\mu_f$ is weakly decomposable on $[\zero,X]$, then $\mu_f$ is weakly decomposable on $[\zero,A]$
    for every $A \in [\zero,X]$. In particular, $\mu_f$ is increasing on $[\zero,X]$.
\end{proposition}
\begin{proof}
    By Corollary~\ref{cor:mu_q-mod_gives_additive}, there exists a maximal chain $\zero\lessdot H_m\lessdot\cdots H_0=X$ with the associated layering $L_1,\dots,L_m$ such that for $a_j\in L_j$ and $S\subseteq[m]$ satisfying $A=\bigvee_{j\in S}a_j$, we have $\mu_f(A)=\sum_{j\in S}f(a_j)$.
    Choose some $S\subseteq[m]$ and $a_j \in L_j$ for each $j \in S$ such that $A=\bigvee_{j\in S}a_j$.
    Let $Y\leq A$. By Lemma~\ref{lem:indep_layers_2_properties}, there exists $T\subseteq S$ and $a_j'\in L_j$ such that $Y=\bigvee_{j\in T}a_j'$. Moreover, $Y$ is a relative complement of $Z=\bigvee_{j\in S\setminus T}a_j$ in $[\zero, A]$. By Corollary~\ref{cor:mu_q-mod_gives_additive} and Property 1 of Lemma~\ref{lem:q-mod_implies_mu-support}, we obtain
    $$\mu_f(Y)+\mu_f(Z)=\sum_{j\in T}f(a')+\sum_{j\in S\setminus T}f(a_j)=\sum_{j\in S}f(a_j)=\mu_f(A).$$
    The result follows.
\end{proof}

\section{Characterizing the Cyclic Flats of an \texorpdfstring{$\mathcal{L}$}{L}-Polymatroid}\label{sec:main}

In this section, we provide a set of axioms that characterize when a lattice $\mathcal{Z}$ is the lattice of cyclic flats of an $\mL$-polymatroid. The sets of axioms given in \cite{AlfByr22} and \cite{Csirm20} can be recovered from the axioms in Definition \ref{def:axioms} by setting \( f \) as a function taking values in \( \{0,1\} \) and by considering \( \mathcal{L} \) to be a subspace lattice and a Boolean lattice, respectively. Several of the technical results in this section are generalisations of results from \cite[Section 4]{Csirm20}.

\begin{definition}
    A \emph{weighted lattice} is a pair $(\mathcal{Z},\lambda)$ where $\mathcal{Z}$ is a lattice and $\lambda$ is a non-negative real-valued function on $\mathcal{Z}$.
\end{definition}

\begin{notation}
    For the remainder of this section, we let $(\mathcal{Z},\vee_\mathcal{Z},\wedge_\mathcal{Z})$ denote an arbitrary lattice such that $\mathcal{Z}$ is a subset of $\mathcal{L}$. 
    Note that the meet and join operations in $\mathcal{Z}$ may differ from those in $\mathcal{L}$. Complements are always taken in the lattice $\mathcal{L}$. Explicitly, for $Z\in\mathcal{Z}$, we have $Z\in\mathcal{L}$, and for any $Z^c\in\mathbf{C}(Z)$ we have $Z\wedge Z^c=\zero$ and $Z\vee Z^c=\one$. Furthermore, $Z^c$ is not necessarily an element of $\mathcal{Z}$.

    Recall that $f:\mA(\mathcal{L})\rightarrow\mathbb{R}_{\geq0}$ is a fixed non-negative function. For the remainder we let $\lambda:\mathcal{Z}\rightarrow\mathbb{R}_{\geq0}$ denote a fixed non-negative function. We write $(\mathcal{Z},\lambda,f)$ to denote a lattice $\mathcal{Z}$ (as a subset of $\mathcal{L}$) endowed with the functions $\lambda$ and $f$.
\end{notation}

\begin{definition}
    For any $X\in\mathcal{L}$, $Z\in\mathcal{Z}$, and $Z^c \in \cC(Z)$, we define:
    \begin{enumerate}
        \item $\rho_{(\lambda,f)}(X;Z^c,Z) = \lambda(Z) + \mu_f(X\wedge Z^c)$,
        \item $\rho_{(\lambda,f)}(X) = \min\{\rho_{(\lambda,f)}(X;Z^c,Z) : Z\in\mathcal{Z}, Z^c\in \mathbf{C}(Z;X)\},$
        \item $\mathcal{Z}(X) = \{Z\in\mathcal{Z} : \rho_{(\lambda,f)}(X;Z^c,Z) = \rho_{(\lambda,f)}(X) \text{ for some } Z^c \in \mathbf{C}(Z;X) \}$.
    \end{enumerate}
\end{definition}

The following axioms apply to an arbitrary weighted lattice $(\mathcal{Z},\lambda)$, where $\mathcal{Z}$ is a subset of $\mathcal{L}$, and the atoms of $\mathcal{L}$ are assigned weights from the (arbitrary) function $f$.
We will see that these axioms precisely characterize when $(\mathcal{L},\rho_{(\lambda,f)})$ is an $\mL$-polymatroid with cyclic flats $\mathcal{Z}$. 

\begin{definition}\label{def:axioms}
    We call the following axioms the \emph{cyclic flat axioms}.
\begin{description}
    \item[\namedlabel{z1}{{\rm (Z1)}}] For all $X\in\mathcal{L}$, there exists $Z\in\mathcal{Z}(X)$ such that:
    \begin{description}
        \item[\namedlabel{z1i}{{\rm (i)}}] For all $Z^c\in\mathbf{C}(Z;X)$, we have $\rho_{(\lambda,f)}(X;Z^c,Z)=\rho_{(\lambda,f)}(X).$
        \item[\namedlabel{z1ii}{{\rm (ii)}}] The function $\mu_f$ is weakly decomposable on $[\zero,X\wedge Z^c]$ for all $Z^c\in\mathbf{C}(Z;X)$.
    \end{description}
    \item[\namedlabel{z2}{{\rm (Z2)}}] For any $Z_1,Z_2\in\mathcal{Z}$, any $(Z_1\wedge_\mathcal{Z} Z_2)^c\in \mathbf{C}(Z_1\wedge_\mathcal{Z}Z_2)$, and any $A\in\mathcal{L}$, we have 
    \[
    \lambda(Z_1)+\lambda(Z_2)\geq\lambda(Z_1\wedge_\mathcal{Z} Z_2)+\lambda(Z_1\vee_\mathcal{Z} Z_2)+\mu_f(A\wedge Z_1\wedge Z_2\wedge(Z_1\wedge_\mathcal{Z} Z_2)^c).
    \]
    \item[\namedlabel{z3}{{\rm (Z3)}}] For any $Z_1, Z_2\in\mathcal{Z}$, if $Z_1\leq Z_2$, then for any $Z_1^c\in\mathbf{C}(Z_1)$, we have
    \[
    \lambda(Z_2)-\lambda(Z_1)\leq\mu_f (Z_2\wedge Z_1^c).
    \]
    \item [\namedlabel{z3*}{{\rm (Z4)}}] For any $Z_1,Z_2\in\mathcal{Z}$, if $Z_1<Z_2$, then for any $H\in\mathcal{H}(Z_2)\cap[Z_1,Z_2]$ there exist $Z_1^c\in\mathbf{C}(Z_1)$ and $H^c\in\mathbf{C}(H)$ such that
    $$0<\lambda(Z_2)-\lambda(Z_1)<\mu_f(H\wedge Z_1^c)+\mu_f(Z_2\wedge H^c).$$
    \item[\namedlabel{z5}{{\rm (Z5)}}] $\lambda(0_\mathcal{Z})=0$.
    \item[\namedlabel{z6}{{\rm (Z6)}}] $\mu_f(a)>0$ for any $a\nleq0_\mathcal{Z}$.
\end{description}
\end{definition}

To prove that these cyclic flat axioms characterize all $(\mathcal{Z},\lambda,f)$ that coincide with the lattice of cyclic flats of an $\mL$-polymatroid $(\mathcal{L},\rho_{(\lambda,f)})$, we begin with Proposition~\ref{prop:q-poly_sats_axioms}, which is the easier of the two directions of implication in the main result of this paper, Theorem~\ref{th:main}.

\begin{proposition}\label{prop:q-poly_sats_axioms}
    Let $\mathcal{Z}$ be the lattice of cyclic flats of the $\mL$-polymatroid $\mP=(\mathcal{L},r)$. Then $(\mZ,r,r)$ satisfies the cyclic flat axioms \ref{z1}--\ref{z6}.
\end{proposition}

\begin{proof}
Let $\mathcal{Z}\subseteq\mathcal{L}$ be the lattice of cyclic flats of $(\mathcal{L},r)$. 
By Theorem~\ref{thm:convol_identity}, \ref{z1}\ref{z1i} holds. By Lemma~\ref{lem:mu_is_q-mod_on_cf}, \ref{z1}\ref{z1ii} holds.
By the submodularity of $r$, for any $Z_1,Z_2\in\mathcal{Z}$, we have that
    $$r(Z_1)+r(Z_2)\geq r(Z_1\vee Z_2)+r(Z_1\wedge Z_2).$$
    Since $Z_1\vee_\mathcal{Z}Z_2=\textup{cl}(Z_1\vee Z_2)$, it follows that $r(Z_1\vee Z_2)=r(Z_1\vee_\mathcal{Z}Z_2)$. Since $Z_1\wedge_\mathcal{Z}Z_2=\textup{cyc}(Z_1\wedge Z_2)$, by Lemma~\ref{lem:cyc_rank_identity}, we have that
    $$r(Z_1\wedge Z_2)=r(Z_1\wedge_\mathcal{Z}Z_2)+\mu_r(Z_1\wedge Z_2\wedge (Z_1\wedge_\mathcal{Z}Z_2)^c)$$ 
    for any complement $(Z_1\wedge_\mathcal{Z}Z_2)^c \in \mathbf{C}(Z_1\wedge_\mathcal{Z}Z_2)$. Axiom \ref{z2} follows.

    Let $Z_1,Z_2 \in \mathcal{Z}$ with $Z_1 \leq Z_2$. Then for any $Z_1^c \in \mathbf{C}(Z_1)$, we have 
        $$r(Z_2)-r(Z_1) \leq r(Z_1^c \wedge Z_2) \leq \mu_r(Z_1^c \wedge Z_2)$$
        where both inequalities follow from the submodularity of $r$. Axiom \ref{z3} follows.

        Let $Z_1,Z_2 \in \mathcal{Z}$ with $Z_1 < Z_2$. Clearly $r(Z_1)<r(Z_2)$ since both are flats of $(\mathcal{L},r)$.
        Suppose that \ref{z3*} doesn't hold for $Z_1,Z_2$; that is, suppose there is some $H\in\mathcal{H}(Z_2)\cap[Z_1,Z_2]$ such that
        \begin{align}\label{ineq:z4_proof}
            r(Z_2)-r(Z_1)&\geq\mu_r(H\wedge Z_1^c)+\mu_r(Z_2\wedge H^c) 
        \end{align}
        for any $Z_1^c\in\mathbf{C}(Z_1)$, and any $H^c\in\mathbf{C}(H)$. Now $H\wedge Z_1=Z_1=Z_2\wedge Z_1$.
        Moreover, using Proposition~\ref{prop:compdecomp}, we obtain, for arbitrary $Z_1^c$,
        $$H=(H\wedge Z_1)\dot\vee(H\wedge Z_1^c)=(Z_2\wedge Z_1)\dot\vee(H\wedge Z_1^c)\quad\textup{and}\quad Z_2=(Z_2\wedge Z_1)\dot\vee(Z_2\wedge Z_1^c).$$
        Using the modularity of $\mathcal{L}$, and the fact that $H\lessdot Z_2$, we thus obtain
        $$1=\textup{h}(Z_2)-\textup{h}(H)=\textup{h}(Z_2\wedge Z_1)+\textup{h}(Z_2\wedge Z_1^c)-\textup{h}(Z_2\wedge Z_1)-\textup{h}(H\wedge Z_1^c),$$ 
        so $\textup{h}(Z_2\wedge Z_1^c)-\textup{h}(H\wedge Z_1^c)=1$ and thus $H\wedge Z_1^c\lessdot Z_2\wedge Z_1^c.$ From $H\lessdot Z_2$, we get $\zero\lessdot Z_2\wedge H^c$. Also, it is clear that $Z_2\wedge H^c\wedge H\wedge Z_1^c=\zero$. Therefore, for any $H^c\in\mathbf{C}(H)$ such that $H^c\leq Z_1^c$, we have
        \begin{equation}\label{eq:Z2Z1H^c}
            Z_2\wedge Z_1^c=(H\wedge Z_1^c)\dot\vee(Z_2\wedge H^c).
        \end{equation}
        Since $Z_1\leq H$, we have that $Z_1\wedge H^c=\zero$ for any $H^c\in\mathbf{C}$. By Corollary~\ref{cor:extend_to_complement}, we thus have for any $H^c\in\mathbf{C}$, that there exists $Z_1^c\in\mathbf{C}(Z_1)$ such that $H^c\leq Z_1^c$. Therefore, we deduce that for any $H^c\in\mathbf{C}(H)$, there exists $Z_1^c\in\mathbf{C}(Z_1)$ such that (\ref{eq:Z2Z1H^c}) holds.

        Since $Z_2\wedge H^c\leq Z_1^c$ and $Z_2\wedge H^c$ is an atom, we have $Z_2\wedge H^c\nleq Z_1$ and thus $Z_2\wedge H^c\nleq 0_\mathcal{Z}=\textup{cl}(\zero)$. Therefore, using the fact that $Z_2\wedge H^c$ is an atom, we have
        $$\mu_r(Z_2\wedge H^c)=r(Z_2\wedge H^c)>0.$$ 
        By Theorem~\ref{thm:convol_identity}, we have $r(H)\leq r(Z_1)+\mu_r(H\wedge Z_1^c)$. We thus obtain, with use of (\ref{ineq:z4_proof}),
        \begin{align}
            r(Z_2)-r(H)&\geq r(Z_2)-r(Z_1)-\mu_r(H\wedge Z_1^c)\nonumber\\
            &\geq\mu_r(Z_2\wedge Z_1^c)+\mu_r(Z_2\wedge H^c)-\mu_r(Z_2\wedge Z_1^c)\nonumber\\
            &=\mu_r(Z_2\wedge H^c).\label{eq:Z4_proof}
        \end{align}
        Recall that $H^c\in\mathbf{C}(H)$ was chosen arbitrarily. Therefore, by definition, if $r(H)<r(Z_2)$, then (\ref{eq:Z4_proof}) gives us that $Z_2$ is not cyclic.

        Suppose then that $r(H)=r(Z_2)$. By (\ref{eq:Z4_proof}), we then have that 
        $0=\mu_r(Z_2\wedge H^c)=r(Z_2\wedge H^c),$
        which contradicts the earlier assertion that $r(Z_2\wedge H^c)>0$. Therefore, \ref{z3*} must hold. Axiom \ref{z5} follows immediately from the fact that $\textup{cl}(\zero)=0_\mathcal{Z}$. Axiom \ref{z6} follows immediately from the fact that if $a \nleq \textup{cl}(\zero)$, then $r(a)>0$.
\end{proof}

We state in Theorem~\ref{th:main} the main result of this section, which is the converse of Proposition~\ref{prop:q-poly_sats_axioms}.

\begin{theorem}\label{th:main}
    If $(\mathcal{Z},\lambda,f)$ satisfies the cyclic flat axioms \ref{z1}--\ref{z6}, then $\mathcal{Z}$ is the lattice of cyclic flats of the $\mathcal{L}$-polymatroid $(\mathcal{L},\rho_{(\lambda,f)})$.
\end{theorem}

We begin with a lemma.

\begin{lemma}\label{lem:two_cases}
    Let $X,Y \in \mathcal{L}$ and $a \in \mA(\mathcal{L}) \backslash \mA(Y)$. Let $X^c \in \mathbf{C}(X;Y)$. Then there exists $X^{\hat{c}} \in \mathbf{C}(X; Y \dot\vee a)$ such that $Y \wedge X^c = Y \wedge X^{\hat{c}}$, and one of the following is true:
    \begin{enumerate}
        \item $(Y \dot\vee a) \wedge X^{\hat{c}} = Y \wedge X^c$; or
        \item $a \leq X^{\hat{c}}$, in particular $(Y \dot\vee a) \wedge X^{\hat{c}} = (Y \wedge X^c) \dot\vee a$.
    \end{enumerate}
\end{lemma}

\begin{proof}
    If $X\wedge((Y\wedge X^c)\dot\vee a)=\zero$, then by Lemma~\ref{lem:dir_join_assoc}, there exists $X^{\hat{c}}\in\mathbf{C}(X)$ such that $(Y\wedge X^c)\dot\vee a\leq X^{\hat{c}}$ (and thus $a\leq X^{\hat{c}}$). By the modularity of $\mathcal{L}$, we then have $Y\wedge X^c=Y\wedge X^{\hat{c}}$, and we deduce that $X^{\hat{c}}\in\mathbf{C}(X;Y\dot\vee a)$, so Property 2 of the lemma is satisfied.

    If instead $X\wedge((Y\wedge X^c)\dot\vee a)\neq\zero$, the modularity of $\mathcal{L}$ means $X\wedge((Y\wedge X^c)\dot\vee a)=a'$ for some $a'\in\mathcal{A}(X)$. Therefore, $a'\nleq Y\wedge X^c$, and so $Y\wedge X^c<(Y\wedge X^c)\dot\vee a'\leq(Y\wedge X^c)\dot\vee a$, which means $(Y\wedge X^c)\dot\vee a'=(Y\wedge X^c)\dot\vee a$ since $Y\wedge X^c\lessdot(Y\wedge X^c)\dot\vee a$. We thus deduce that $a'\nleq Y$ and so $Y\dot\vee a=Y\dot\vee a'$. This means that $(Y\dot\vee a)\wedge X=(Y\dot\vee a')\wedge X=(Y\wedge X)\dot\vee a'$. Observe, by the modularity of $\mathcal{L}$, that $\textup{h}(Y\dot\vee a)=\textup{h}(Y)+1$, $\textup{h}((Y\dot\vee a)\wedge X)=\textup{h}(Y\wedge X)+1$, and recall that $X^c\in\mathbf{C}(X;Y)$. We thus obtain
    $$\textup{h}(Y\dot\vee a)\geq\textup{h}((Y\dot\vee a)\wedge X)+\textup{h}((Y\dot\vee a)\wedge X^c)\geq 1+\textup{h}(Y\wedge X)+\textup{h}(Y\wedge X^c)=\textup{h}(Y\dot\vee a),$$
    which means that $\textup{h}((Y\dot\vee a)\wedge X^c)=\textup{h}(Y\wedge X^c)$ and thus $(Y\dot\vee a)\wedge X^c=Y\wedge X^c$. Therefore, we let $X^c=X^{\hat{c}}$ and Property 1 of the lemma is satisfied.
\end{proof}

Note, from Lemma~\ref{lem:two_cases}, it is clear that $X^{\hat{c}}\in\mathbf{C}(X;Y)\cap\mathbf{C}(X;Y\dot\vee a)$ since $X^c\in\mathbf{C}(X;Y)$ and $Y\wedge X^c=Y\wedge X^{\hat{c}}$.

In the following, Theorem~\ref{lem:q-poly_mu_ineq} is a technical result made in preparation for Theorem~\ref{thm:axioms_imply_q-poly}, in which we show that if axioms \ref{z1}, \ref{z2}, and \ref{z5} hold for $(\mZ, \lambda, f)$, then $\rho_{(\lambda,f)}$ is the rank function of an $\mL$-polymatroid.

\begin{theorem}\label{lem:q-poly_mu_ineq}
    Suppose that $(\mZ, \lambda, f)$ satisfies \ref{z1}. Let $X \in \mathcal{L}$ and let $a_1, a_2 \in \mA(\mathcal{L}) \backslash \mA(X)$ such that $\len([X,X\dot\vee a_1\dot\vee a_2])=2$. Let
    $Z_1 \in \mathcal{Z}(X \dot\vee a_1)$ and $Z_2 \in \mathcal{Z}(X \dot\vee a_2)$ both satisfy \ref{z1} for $X\dot\vee a_1$ and $X\dot\vee a_2$ respectively. 
    Let $Z_k^c \in \mathbf{C}(Z_k; X \dot\vee a_k)$ for $k = 1, 2$. We have
    \begin{equation}\label{eq:mu_q-poly_ineq}
        \mu_f((X \dot\vee a_1) \wedge Z_1^c) + \mu_f((X \dot\vee a_2) \wedge Z_2^c) \geq \mu_f(X \wedge (Z_1 \wedge Z_2)^c) + \mu_f((X \dot\vee a_1 \dot\vee a_2) \wedge (Z_1 \vee_\mathcal{Z} Z_2)^c)
    \end{equation}
    for some $(Z_1 \vee_\mathcal{Z} Z_2)^c \in \mathbf{C}(Z_1 \vee_\mathcal{Z} Z_2; X \dot\vee a_1 \dot\vee a_2)$ and some $(Z_1 \wedge Z_2)^c \in \mathbf{C}(Z_1 \wedge Z_2; X)$.
\end{theorem}

\begin{proof}    
    By assumption, for \( k = 1, 2 \), we have \( Z_k \in \mathcal{Z}(X \dot\vee a_k) \), and \ref{z1} holds for each. By \ref{z1}\ref{z1i}, this means that for $k=1,2$ we have
\[
\rho_{(\lambda, f)}(X \dot\vee a_k; Z_k^c, Z_k) = \rho_{(\lambda, f)}(X \dot\vee a_k)
\]
for all \( Z_k^c \in \textbf{C}(Z_k; X \dot\vee a_k) \). By \ref{z1}\ref{z1ii}, we have that for $k=1,2$, the function $\mu_f$ is weakly decomposable on $[\zero,(X\dot\vee a_k)\wedge Z_k^c]$ for any $Z_k^c\in\mathbf{C}(Z_k;X\dot\vee a_k)$. Moreover, by Proposition~\ref{cor:mu_is_increasing}, we thus have for $k=1,2$ that $\mu_f$ is increasing on $[\zero,(X\dot\vee a_k)\wedge Z_k^c]$ for any $Z_k^c\in\mathbf{C}(Z_k)$. Using Lemma~\ref{lem:two_cases}, we choose \( Z_1^c \in \mathbf{C}(Z_1; X \dot\vee a_1) \cap \mathbf{C}(Z_1; X) \) such that either
\begin{equation}\label{eq:Z_1^c_either_or}
    (X \dot\vee a_1) \wedge Z_1^c = X \wedge Z_1^c, \quad \text{or} \quad a_1 \leq Z_1^c.
\end{equation}

In the case of $a_1\leq Z_1^c$, we have that $X\wedge Z_1^c\lessdot(X\dot\vee a_1)\wedge Z_1^c$ by the modularity of $\mathcal{L}$. Since $\mu_f$ is weakly decomposable on the interval $[\zero,(X\dot\vee a_1)\wedge Z_1^c]$, we have
\begin{equation*}
    \mu_f((X\dot\vee a_1)\wedge Z_1^c)=\mu_f(X\wedge Z_1^c)+f(a_1')
\end{equation*} 
for some $a_1'\in\mathcal{A}((X\dot\vee a_1)\wedge Z_1^c)\setminus\mathcal{A}(X\wedge Z_1^c)$. It is clear that $X\dot\vee a_1=X\dot\vee a_1'$. We observe that if the inequality (\ref{eq:mu_q-poly_ineq}) holds true for some $a_1$ as described in the theorem statement, then it holds true for any $a''_1\in\mathcal{A}(\mathcal{L})$ satisfying $X \vee a_1=X \vee a''_1$. Therefore, in the case of $a_1\leq Z_1^c$, we will assume, without loss of generality, that
\begin{equation}\label{eq:a_1_fix}
    \mu_f((X\dot\vee a_1)\wedge Z_1^c)=\mu_f(X\wedge Z_1^c)+f(a_1).
\end{equation} 

In the following, we construct a complement of $Z_1\wedge Z_2$ that suits our purposes.
By Corollary~\ref{cor:decomp_both}, there exists \( Z_2^{\tilde{c}} \in \textbf{C}(Z_2; Z_1) \cap \textbf{C}(Z_2; X \wedge Z_1) \). Using the modular law on \( \mathcal{L} \), we obtain
\begin{align*}
    (Z_1^c \dot\vee (Z_1 \wedge Z_2^{\tilde{c}})) \wedge Z_1 \wedge Z_2
&= ((Z_1^c \dot\vee (Z_1 \wedge Z_2^{\tilde{c}})) \wedge Z_1) \wedge Z_2\\
&= ((Z_1^c\wedge Z_1)\vee (Z_1 \wedge Z_2^{\tilde{c}})) \wedge Z_2\\
&= (Z_1 \wedge Z_2^{\tilde{c}}) \wedge Z_2\\
&=\zero.
\end{align*}
By Corollary~\ref{cor:extend_to_complement}, this gives us that \( Z_1^c \dot\vee (Z_1 \wedge Z_2^{\tilde{c}}) \leq (Z_1 \wedge Z_2)^c \) for some \( (Z_1 \wedge Z_2)^c \in \textbf{C}(Z_1 \wedge Z_2) \). Therefore, using the modularity of $\mathcal{L}$, we deduce that the following are equivalent:
\begin{itemize}
    \item $Z_1^c \dot\vee (Z_1 \wedge Z_2^{\tilde{c}}) \in \textbf{C}(Z_1 \wedge Z_2)$
    \item $\text{h}((Z_1 \wedge Z_2)^c) - \text{h}(Z_1^c) = \text{h}(Z_1 \wedge Z_2^{\tilde{c}})$
    \item $\text{h}(Z_1) - \text{h}(Z_1 \wedge Z_2) = \text{h}(Z_1 \wedge Z_2^{\tilde{c}})$
    \item $Z_2^{\tilde{c}} \in \textbf{C}(Z_2; Z_1).$
\end{itemize} In particular, we have that \( Z_1^c \dot\vee (Z_1 \wedge Z_2^{\tilde{c}}) \in \textup{\textbf{C}}(Z_1 \wedge Z_2) \).
Since \( Z_1^c \in \mathbf{C}(Z_1; X) \) and \( Z_2^{\tilde{c}} \in \mathbf{C}(Z_2; X \wedge Z_1) \), we obtain the following:
\begin{align}\label{eq:split_w}
    X &= (X \wedge Z_1^c) \dot\vee (X \wedge Z_1) \nonumber \\
    &= (X \wedge Z_1^c) \dot\vee (X \wedge Z_1 \wedge Z_2^{\tilde{c}}) \dot\vee (X \wedge Z_1 \wedge Z_2) \nonumber \\
    &\leq (X \wedge (Z_1^c \dot\vee (Z_1 \wedge Z_2^{\tilde{c}}))) \dot\vee (X \wedge Z_1 \wedge Z_2)\quad\textup{(by Lemma~\ref{lem:gen_latt_distrib})} \nonumber \\
    &\leq X.
\end{align}
Clearly, we must have equality in (\ref{eq:split_w}). Given that \( Z_1^c \dot\vee (Z_1 \wedge Z_2^{\tilde{c}}) \in \textup{\textbf{C}}(Z_1 \wedge Z_2) \), we have thus shown \( Z_1^c \dot\vee (Z_1 \wedge Z_2^{\tilde{c}}) \in \textup{\textbf{C}}(Z_1 \wedge Z_2; X) \). We now have our desired complement
$$(Z_1\wedge Z_2)^c = Z_1^c\vee(Z_1\wedge Z_2^{\tilde{c}}).$$ 
Furthermore, from (\ref{eq:split_w}), we deduce that
\[
\textup{h}(X \wedge (Z_1^c \dot\vee (Z_1 \wedge Z_2^{\tilde{c}}))) = \textup{h}((X \wedge Z_1^c) \dot\vee (X \wedge Z_1 \wedge Z_2^{\tilde{c}})).
\]
From Lemma~\ref{lem:gen_latt_distrib}, we have that
$$X \wedge (Z_1^c \dot\vee (Z_1 \wedge Z_2^{\tilde{c}})) \geq (X \wedge Z_1^c) \dot\vee (X \wedge Z_1 \wedge Z_2^{\tilde{c}}).$$
We thus obtain
\[
X \wedge (Z_1^c \dot\vee (Z_1 \wedge Z_2^{\tilde{c}})) = (X \wedge Z_1^c) \dot\vee (X \wedge Z_1 \wedge Z_2^{\tilde{c}}),
\]
which, by Lemma~\ref{lem:mu_submod}, gives us
\begin{equation}\label{eq:mu_meet_ineq}
    \mu_f(X \wedge (Z_1^c \dot\vee (Z_1 \wedge Z_2^{\tilde{c}}))) \leq \mu_f(X \wedge Z_1^c) + \mu_f(X \wedge Z_1 \wedge Z_2^{\tilde{c}}).
\end{equation}

Note that, by definition, $Z_1\vee Z_2\leq Z_1\vee_\mathcal{Z}Z_2$. In the following, we will construct a complement of $Z_1\vee_\mathcal{Z}Z_2$ to fit our purposes. 
Since $(Z_1\wedge Z_2^{\tilde{c}})\vee Z_2\leq Z_1\vee_\mathcal{Z}Z_2$, we have that $((Z_1\wedge Z_2^{\tilde{c}})\vee Z_2)\wedge(Z_1\vee_\mathcal{Z}Z_2)^c=\zero$ for any $(Z_1\vee_\mathcal{Z}Z_2)^c\in\mathbf{C}(Z_1\vee_\mathcal{Z}Z_2)$. Therefore, by Lemma~\ref{lem:dir_join_assoc}, we have that 
$$Z_2\wedge ((Z_1\wedge Z_2^{\tilde{c}})\dot\vee(Z_1\vee_\mathcal{Z}Z_2)^c)=\zero$$
for any $(Z_1\vee_\mathcal{Z}Z_2)^c\in\mathbf{C}(Z_1\vee_\mathcal{Z}Z_2)$. By Corollary~\ref{cor:extend_to_complement}, for each $(Z_1\vee_\mathcal{Z}Z_2)^c\in\mathbf{C}(Z_1\vee_\mathcal{Z}Z_2)$ there exists $Z_2^{c'}\in\mathbf{C}(Z_2)$ such that $(Z_1\wedge Z_2^{\tilde{c}})\dot\vee(Z_1\vee_\mathcal{Z}Z_2)^c\leq Z_2^{c'}$. Therefore, by considering $(Z_1\vee_\mathcal{Z}Z_2)^c\in\mathbf{C}(Z_1\vee_\mathcal{Z}Z_2;X\dot\vee a_2)$, we can choose $Z_2^{c'}\in\mathbf{C}(Z_2)$ such that $Z_1\wedge Z_2^{\tilde{c}}\leq Z_2^{c'}$ and 
\begin{equation}
    X\dot\vee a_2=((X\dot\vee a_2)\wedge (Z_1\vee_\mathcal{Z}Z_2))\vee((X\dot\vee a_2)\wedge Z_2^{c'}).
\end{equation}
Furthermore, since $Z_2^{\tilde{c}}\in\mathbf{C}(Z_2;Z_1)$, we deduce from $Z_1\wedge Z_2^{\tilde{c}}\leq Z_2^{c'}$ that $Z_2^{c'}\in\mathbf{C}(Z_2;Z_1)$, from which we obtain (using the modularity of $\mathcal{L}$)
$$\textup{h}(Z_1\wedge Z_2^{\tilde{c}})=\textup{h}(Z_1\wedge Z_2^{c'}).$$
We thus deduce that $Z_1\wedge Z_2^{\tilde{c}}=Z_1\wedge Z_2^{c'}$. By our construction of $(Z_1\wedge Z_2)^c$, we then have
$$(Z_1\wedge Z_2)^c=Z_1^c\dot\vee(Z_1\wedge Z_2^{\hat{c}})=Z_1^c\dot\vee(Z_1\wedge Z_2^{c'}).$$
Moreover, by (\ref{eq:mu_meet_ineq}), we obtain
\begin{equation}\label{eq:mu_z2c'}
    \mu_f(X \wedge (Z_1^c \dot\vee (Z_1 \wedge Z_2^{c'}))) \leq \mu_f(X \wedge Z_1^c) + \mu_f(X \wedge Z_1 \wedge Z_2^{c'}).
\end{equation}

Now choose an arbitrary $Z_1^{c'}\in\mathbf{C}(Z_1;(X\dot\vee a_2)\wedge Z_2^{c'})$. We then obtain
\begin{align}
    X \dot\vee a_2 &= ((X \dot\vee a_2) \wedge (Z_1 \vee_\mathcal{Z} Z_2)) \vee ((X \dot\vee a_2) \wedge Z_2^{c'})\nonumber \\
    &= ((X \dot\vee a_2) \wedge (Z_1 \vee_\mathcal{Z} Z_2)) \vee ((X \dot\vee a_2)\wedge Z_1 \wedge Z_2^{c'}) \vee ((X \dot\vee a_2)\wedge Z_1^{c'} \wedge Z_2^{c'})\nonumber \\
    &= ((X \dot\vee a_2) \wedge (Z_1 \vee_\mathcal{Z} Z_2)) \vee ((X \dot\vee a_2)\wedge Z_1^{c'} \wedge Z_2^{c'}),\label{eq:max_Y_rel_comp}
\end{align}
where the equality in (\ref{eq:max_Y_rel_comp}) comes from $(X\dot\vee a_2)\wedge Z_1\wedge Z_2^{c'}\leq(X\dot\vee a_2)\wedge (Z_1\vee_\mathcal{Z}Z_2)$.

Let $Y\in\mathcal{L}$ be maximal in $[\zero,(X\dot\vee a_2)\wedge Z_1^{c'}\wedge Z_2^{c'}]$ such that
$$(X\dot\vee a_2)\wedge(Z_1\vee_\mathcal{Z}Z_2)\wedge Y=\zero.$$
If $X\dot\vee a_2>((X\dot\vee a_2)\wedge(Z_1\vee_\mathcal{Z}Z_2))\dot\vee Y$, then by (\ref{eq:max_Y_rel_comp}), there must exist $x\in\mathcal{A}((X\dot\vee a_2)\wedge Z_1^{c'}\wedge Z_2^{c'})$ such that $(((X\dot\vee a_2)\wedge(Z_1\vee_\mathcal{Z}Z_2))\dot\vee Y)\wedge x=\zero$. By Lemma~\ref{lem:dir_join_assoc}, we thus contradict the maximality of $Y$.
Therefore, we deduce that this maximality of $Y$ ensures
\begin{equation}\label{eq:Xa2Y}
    X\dot\vee a_2=((X\dot\vee a_2)\wedge(Z_1\vee_\mathcal{Z}Z_2))\dot\vee Y.
\end{equation}
Therefore, by Corollary~\ref{cor:extend_to_complement}, we can extend $Y$ to a complement of $Z_1\vee_\mathcal{Z}Z_2$. By (\ref{eq:Xa2Y}), we have that such a complement is an element of $\mathbf{C}(Z_1\vee_\mathcal{Z}Z_2;X\dot\vee a_2)$, from which it is clear that the meet of such a complement with $X\dot\vee a_2$ is equal to $Y$. In other words, for any choice of $Z_1^{c'}\in\mathbf{C}(Z_1;(X\dot\vee a_2)\wedge Z_2^{c'})$, there exists $(Z_1\vee_\mathcal{Z}Z_2)^c\in\mathbf{C}(Z_1\vee_\mathcal{Z}Z_2;X\dot\vee a_2)$ such that
\begin{equation}\label{ineq:z1z2z1^c}
    (X\dot\vee a_2)\wedge (Z_1\vee_\mathcal{Z}Z_2)^c\leq(X\dot\vee a_2)\wedge Z_1^{c'}\wedge Z_2^{c'}.
\end{equation}

Recall that our choice of $Z_1^{c'}\in\mathbf{C}(Z_1;(X\dot\vee a_2)\wedge Z_2^{c'})$ above was arbitrary. We will next proceed to fix a $Z_1^{c'}\in\mathbf{C}(Z_1;(X\dot\vee a_2)\wedge Z_2^{c'})$, after which we will fix a $(Z_1\vee_\mathcal{Z}Z_2)^c\in\mathbf{C}(Z_1\vee_\mathcal{Z}Z_2)$.
By \ref{z1}\ref{z1ii}, $\mu_f$ is weakly decomposable on the interval $[\zero,(X\dot\vee a_2)\wedge Z_2^{c'}]$. Therefore, there is a relative complement $U$ of $(X\dot\vee a_2)\wedge Z_1\wedge Z_2^{c'}$ in the interval $[\zero,(X\dot\vee a_2)\wedge Z_2^{c'}]$ such that 
\begin{equation}\label{eq:U_quasi-mod}
    \mu_f((X\dot\vee a_2)\wedge Z_2^{c'})=\mu_f((X\dot\vee a_2)\wedge Z_1\wedge Z_2^{c'})+\mu_f(U).
\end{equation}
Since $U\leq(X\dot\vee a_2)\wedge Z_2^{c'}$ and $U\wedge(X\dot\vee a_2)\wedge Z_1\wedge Z_2^{c'}=\zero$, we deduce that $U\wedge Z_1=\zero$. By Corollary~\ref{cor:extend_to_complement}, we can extend $U$ to some $Z_1^{c'}\in\mathbf{C}(Z_1)$, from which we deduce
\begin{align*}
    (X\dot\vee a_2)\wedge Z_2^{c'}&=((X\dot\vee a_2)\wedge Z_1\wedge Z_2^{c'})\dot\vee U\\
    &\leq ((X\dot\vee a_2)\wedge Z_1\wedge Z_2^{c'})\dot\vee((X\dot\vee a_2)\wedge Z_1^{c'}\wedge Z_2^{c'})\\
    &\leq (X\dot\vee a_2)\wedge Z_2^{c'},
\end{align*}
and hence we have equality. It follows that $\textup{h}(U)=\textup{h}((X\dot\vee a_2)\wedge Z_1^{c'}\wedge Z_2^{c'})$, and thus we deduce $U=(X\dot\vee a_2)\wedge Z_1^{c'}\wedge Z_2^{c'}$. Furthermore, by the above equalities, we obtain that $Z_1^{c'}\in\mathbf{C}(Z_1;(X\dot\vee a_2)\wedge Z_2^{c'})$. We now fix such a $Z_1^{c'}\in\mathbf{C}(Z_1;(X\dot\vee a_2)\wedge Z_2^{c'})$, which is to say, replacing $U$ with $(X\dot\vee a_2)\wedge Z_1^{c'}\wedge Z_2^{c'}$ in (\ref{eq:U_quasi-mod}), we get
\begin{equation}\label{eq:z1'_split}
    \mu_f((X\dot\vee a_2)\wedge Z_2^{c'})=\mu_f((X\dot\vee a_2)\wedge Z_1\wedge Z_2^{c'})+\mu_f((X\dot\vee a_2)\wedge Z_1^{c'}\wedge Z_2^{c'}).
\end{equation}
We now fix $(Z_1\vee_\mathcal{Z}Z_2)^c\in\mathbf{C}(Z_1\vee_\mathcal{Z}Z_2;X\dot\vee a_2)$ such that (\ref{ineq:z1z2z1^c}) is satisfied. Recall that $\mu_f$ is increasing on $[\zero,(X\dot\vee a_2)\wedge Z_2^{c'}]$. Therefore, we obtain
\begin{equation}\label{eq:eim_request}
    \mu_f((X\dot\vee a_2)\wedge(Z_1\vee_\mathcal{Z}Z_2)^c)\leq\mu_f((X\dot\vee a_2)\wedge Z_1^{c'}\wedge Z_2^{c'}).
\end{equation}

By Lemma~\ref{lem:two_cases}, there exists $(Z_1 \vee_\mathcal{Z} Z_2)^{\hat{c}} \in \mathbf{C}(Z_1 \vee_\mathcal{Z} Z_2; X \dot\vee  a_2)\cap \mathbf{C}(Z_1 \vee_\mathcal{Z} Z_2; X \dot\vee a_1 \dot\vee a_2)$ such that $(Z_1 \vee_\mathcal{Z} Z_2)^{\hat{c}} \wedge (X \dot\vee a_2) = (Z_1 \vee_\mathcal{Z} Z_2)^{c} \wedge (X \dot\vee a_2)$ and one of the following holds:
\begin{description}
    \item[\namedlabel{case1}{\textbf{Case 1}}] $a_1\leq(Z_1\vee_\mathcal{Z} Z_2)^{\hat{c}}$; or
    \item[\namedlabel{case2}{\textbf{Case 2}}] $(Z_1 \vee_\mathcal{Z} Z_2)^{\hat{c}} \wedge (X \dot\vee a_1 \dot\vee a_2) = (Z_1 \vee_\mathcal{Z} Z_2)^{c} \wedge (X \dot\vee a_2)$.
\end{description}
We proceed now to consider each of these cases.

Suppose first that we are in \ref{case1}. By the modular law on $\mathcal{L}$, we have
\begin{equation}\label{eq:a_1_in_(Z_1wedgeZ_2)^c}
    (Z_1 \vee_\mathcal{Z} Z_2)^{\hat{c}} \wedge (X \dot\vee a_1 \dot\vee a_2) = ((Z_1 \vee_\mathcal{Z} Z_2)^{c} \wedge (X \dot\vee a_2)) \dot\vee a_1.
\end{equation}
Since $(Z_1 \vee_\mathcal{Z} Z_2)^c \in \textup{\textbf{C}}(Z_1 \vee_\mathcal{Z} Z_2; X \dot\vee a_2)$, we can observe that
\begin{align}
    X \dot\vee a_1 \dot\vee a_2 &= ((X \dot\vee a_2) \wedge (Z_1 \vee_\mathcal{Z} Z_2)) \dot\vee ((X \dot\vee a_2) \wedge (Z_1 \vee_\mathcal{Z} Z_2)^c) \dot\vee a_1 \nonumber\\
    &= ((X \dot\vee a_2) \wedge (Z_1 \vee_\mathcal{Z} Z_2)) \dot\vee ((X \dot\vee a_1 \dot\vee a_2) \wedge (Z_1 \vee_\mathcal{Z} Z_2)^{\hat{c}})\quad\textup{(by (\ref{eq:a_1_in_(Z_1wedgeZ_2)^c}))}\label{eq:z1z2_split}\\
    &\leq ((X\dot\vee a_1 \dot\vee a_2) \wedge (Z_1 \vee_\mathcal{Z} Z_2)) \dot\vee ((X \dot\vee a_1 \dot\vee a_2) \wedge (Z_1 \vee_\mathcal{Z} Z_2)^{\hat{c}})\nonumber\\
    &\leq X \dot\vee a_1 \dot\vee a_2.\nonumber
\end{align}
By the modularity of $\mathcal{L}$, we have $\textup{h}(X\dot\vee a_1\dot\vee a_2)=\textup{h}(X\dot\vee a_2)+1$, given that $a_1$ is an atom. Similarly, using (\ref{eq:a_1_in_(Z_1wedgeZ_2)^c}), since 
$$\textup{h}((X \dot\vee a_1 \dot\vee a_2) \wedge (Z_1 \vee_\mathcal{Z} Z_2)^{\hat{c}})=\textup{h}((X \dot\vee a_2) \wedge (Z_1 \vee_\mathcal{Z} Z_2)^c)+1,$$
we can take the sum of the heights of the terms in (\ref{eq:z1z2_split}) to deduce that 
$$\textup{h}((X \dot\vee a_1 \dot\vee a_2) \wedge (Z_1 \vee_\mathcal{Z} Z_2)) = \textup{h}((X \dot\vee a_2) \wedge (Z_1 \vee_\mathcal{Z} Z_2)),$$ 
from which we get
$$(X \dot\vee a_1 \dot\vee a_2) \wedge (Z_1 \vee_\mathcal{Z} Z_2) = (X \dot\vee a_2) \wedge (Z_1 \vee_\mathcal{Z} Z_2).$$
In this case, we have that $(X \dot\vee a_1) \wedge Z_1 = X \wedge Z_1$. Indeed, since otherwise we would have $(X \dot\vee a_1) \wedge Z_1 = (X \wedge Z_1) \dot\vee a$ for some $a \in \mathcal{A}(Z_1)$. Therefore, since $X \dot\vee a_1 = X \dot\vee a$ and $X \dot\vee a_1 \neq X \dot\vee a_2$, we have $a \nleq X \dot\vee a_2$. Therefore, we have
\begin{align*}
    (X \dot\vee a_1 \dot\vee a_2) \wedge Z_1 &= ((X \dot\vee a_2) \wedge Z_1) \dot\vee a, 
\end{align*}
which gives us
\[
(X \dot\vee a_1 \dot\vee a_2) \wedge Z_1 \nleq (X \dot\vee a_2) \wedge (Z_1 \vee_\mathcal{Z} Z_2) = (X \dot\vee a_1 \dot\vee a_2) \wedge (Z_1 \vee_\mathcal{Z} Z_2),
\]
which yields a contradiction since $(X \dot\vee a_1 \dot\vee a_2) \wedge Z_1 \leq (X \dot\vee a_1 \dot\vee a_2) \wedge (Z_1 \vee_\mathcal{Z} Z_2)$. Therefore, we conclude that $(X \dot\vee a_1) \wedge Z_1 = X \wedge Z_1$. Moreover, since $Z_1^c \in \textup{\textbf{C}}(Z_1; X) \cap \textup{\textbf{C}}(Z_1; X \dot\vee a_1)$, we have
$$X=(X\wedge Z_1)\dot\vee(X\wedge Z_1^c)\quad\textup{and}\quad X\dot\vee a_1=(X\wedge Z_1)\dot\vee((X\dot\vee a_1)\wedge Z_1^c).$$
We are therefore not in the case of $(X\dot\vee a_1)\wedge Z_1^c=X\wedge Z_1^c$, since otherwise we would have $X=X\dot\vee a_1$. By (\ref{eq:Z_1^c_either_or}), we thus conclude that $a_1\leq Z_1^c$. Therefore, by (\ref{eq:a_1_fix}), we have
$$\mu_f((X\dot\vee a_1)\wedge Z_1^c)=\mu_f(X\wedge Z_1^c)+f(a_1).$$

By Lemma~\ref{lem:mu_submod} applied to (\ref{eq:a_1_in_(Z_1wedgeZ_2)^c}), we have 
\begin{equation}\label{eq:vee_submod}
    \mu_f((X\dot\vee a_1\dot\vee a_2)\wedge(Z_1\vee_\mathcal{Z} Z_2)^{\hat{c}})\leq\mu_f((X\dot\vee a_2)\wedge(Z_1\vee_\mathcal{Z} Z_2)^c)+f(a_1).
\end{equation}
Recall that $(Z_1\wedge Z_2)^c=Z_1^c\dot\vee(Z_1\wedge Z_2^{c'})$. We take the sum of the inequalities (\ref{eq:mu_z2c'}) and (\ref{eq:vee_submod}) to get that
\begin{equation}\label{eq:mainZjoinmeet}
    \mu_f(X\wedge (Z_1\wedge Z_2)^c)+\mu_f((X\dot\vee a_1\dot\vee a_2)\wedge (Z_1\vee_\mathcal{Z} Z_2)^{\hat{c}}),
\end{equation}
is bounded above by
\begin{align}\label{eq:too_many_names}
    \mu_f(X\wedge Z_1^c)+\mu_f(X\wedge Z_1\wedge Z_2^{c'})+\mu_f((X\dot\vee a_2)\wedge(Z_1\vee_\mathcal{Z} Z_2)^c)+f(a_1).
\end{align}
By (\ref{eq:a_1_fix}), we have that (\ref{eq:too_many_names}) is equal to
\begin{equation}\label{eq:last_ineq}
    \mu_f((X\dot\vee a_1)\wedge Z_1^c)+\mu_f(X\wedge Z_1\wedge Z_2^{c'})+\mu_f((X\dot\vee a_2)\wedge(Z_1\vee_\mathcal{Z} Z_2)^c).
\end{equation}
By (\ref{eq:eim_request}), we have that (\ref{eq:last_ineq}) is bounded above by
\begin{equation}\label{eq:placeholder}
    \mu_f((X\dot\vee a_1)\wedge Z_1^c)+\mu_f(X\wedge(Z_1\wedge Z_2^{c'}))+\mu_f((X\dot\vee a_2)\wedge Z_1^{c'}\wedge Z_2^{c'}).
\end{equation}
Recall that $\mu_f$ is increasing on $[\zero,(X\dot\vee a_2)\wedge Z_2^{c'}]$. We thus have that $\mu_f(X\wedge Z_1\wedge Z_2^{c'})\leq\mu_f((X\dot\vee a_2)\wedge Z_1\wedge Z_2^{c'})$, which means that (\ref{eq:placeholder}) is bounded above by
\begin{equation}\label{eq:placeholder2}
    \mu_f((X\dot\vee a_1)\wedge Z_1^c)+\mu_f((X\dot\vee a_2)\wedge Z_1\wedge Z_2^{c'})+\mu_f((X\dot\vee a_2)\wedge Z_1^{c'}\wedge Z_2^{c'}).
\end{equation}
By (\ref{eq:z1'_split}), we have that (\ref{eq:placeholder2}) is equal to
\begin{equation}\label{eq:mainZnojoinmeet}
    \mu_f((X\dot\vee a_1)\wedge Z_1^c)+\mu_f((X\dot\vee a_2)\wedge Z_2^{c'}).
\end{equation}
In particular, (\ref{eq:mainZjoinmeet}) is bounded above by (\ref{eq:mainZnojoinmeet}), which establishes the result for \ref{case1}.

Lastly, we address \ref{case2}, which states that
$$(X\dot\vee a_1\dot\vee a_2)\wedge(Z_1\vee_\mathcal{Z} Z_2)^{\hat{c}}=(X\dot\vee a_2)\wedge(Z_1\vee_\mathcal{Z} Z_2)^c.$$
This case is simpler than the previous case as we have 
\begin{equation}\label{eq:last_one?}
    \mu_f((X\dot\vee a_1\dot\vee a_2)\wedge(Z_1\vee_\mathcal{Z} Z_2)^{\hat{c}})=\mu_f((X\dot\vee a_2)\wedge(Z_1\vee_\mathcal{Z} Z_2)^c).
\end{equation}
Recall again that $(Z_1\wedge Z_2)^c=Z_1^c\dot\vee(Z_1\wedge Z_2^{c'})$. We combine (\ref{eq:last_one?}) with (\ref{eq:mu_z2c'}) to get that 
$$\mu_f(X\wedge (Z_1\wedge Z_2)^c)+\mu_f((X\dot\vee a_1\dot\vee a_2)\wedge   (Z_1\vee_\mathcal{Z} Z_2)^{\hat{c}})$$
is bounded above by
\begin{align}\label{eq:last_one}
    \mu_f(X\wedge Z_1^c)+\mu_f(X\wedge Z_1\wedge Z_2^{c'})+\mu_f((X\dot\vee a_2)\wedge(Z_1\vee_\mathcal{Z} Z_2)^c),
\end{align}
which by (\ref{eq:eim_request}), is bounded above by
\begin{align}\label{eq:laster_last_one}
    \mu_f(X\wedge Z_1^c)+\mu_f(X\wedge Z_1\wedge Z_2^{c'})+\mu_f((X\dot\vee a_2)\wedge Z_1^{c'}\wedge Z_2^{c'}).
\end{align}
Again, we observe that $\mu_f$ is increasing on $[\zero,(X\dot\vee a_2)\wedge Z_2^{c'}]$, which means that (\ref{eq:laster_last_one}) is bounded above by
$$\mu_f(X\wedge Z_1^c)+\mu_f((X\dot\vee a_2)\wedge Z_1\wedge Z_2^{c'})+\mu_f((X\dot\vee a_2)\wedge Z_1^{c'}\wedge Z_2^{c'}),$$
which by (\ref{eq:z1'_split}) is equal to
$$\mu_f(X\wedge Z_1^c)+\mu_f((X\dot\vee a_2)\wedge Z_2^{c'}).$$
Since $\mu_f$ is increasing on $[\zero, (X\dot\vee a_1)\wedge Z_1^c]$, we have that $\mu_f(X\wedge Z_1^c)\leq\mu_f((X\dot\vee a_1)\wedge Z_1^c)$. The result follows.
\end{proof}

Our approach to proving that $(\mathcal{L},\rho_{(\lambda,f)})$ is an $\mathcal{L}$-polymatroid will be to verify that, under the conditions \ref{z1}, \ref{z2}, and \ref{z5}, axioms \ref{r1} and \ref{r2} hold, and \ref{r3} holds for all intervals of length 2 in $\mathcal{L}$. The desired result then follows from Corollary~\ref{cor:int2}.

\begin{theorem}\label{thm:axioms_imply_q-poly}
    If $(\mathcal{Z},\lambda,f)$ satisfies \ref{z1}, \ref{z2}, and \ref{z5}, then $(\mathcal{L},\rho_{(\lambda,f)})$ is an $\mathcal{L}$-polymatroid.
\end{theorem}

\begin{proof}
    By \ref{z5}, we have
    \(
        \rho_{(\lambda,f)}(\zero) = \min\{\lambda(Z) + \mu_f(\zero) : Z \in \mathcal{Z} \} = \lambda(\zero_{\mathcal{Z}}) = 0,
    \)
    which verifies \ref{r1}.
    Let $A, B \in \mathcal{L}$ such that $A\leq B$. By \ref{z1}\ref{z1i}, there exists $Z \in \mathcal{Z}$ such that 
    \[
        \rho_{(\lambda,f)}(B) = \lambda(Z) + \mu_f(B \wedge Z^c),
    \] 
    for any $Z^c \in \mathbf{C}(Z;B)$. Let $Z^c$ be an arbitrary element of the set $\mathbf{C}(Z;A)\cap\mathbf{C}(Z;B)$, which by Corollary~\ref{cor:decomp_both}, is non empty. By \ref{z1}\ref{z1ii}, the function $\mu_f$ is weakly decomposable on $[\zero, B \wedge Z^c]$, and by Proposition~\ref{cor:mu_is_increasing}, $\mu_f$ is increasing on that interval. In particular, we have 
    \[
        \mu_f(B \wedge Z^c) \geq \mu_f(A \wedge Z^c).
    \] 
    Thus, we obtain  
    \[
        \rho_{(\lambda,f)}(B) = \lambda(Z) + \mu_f(B \wedge Z^c) \geq \lambda(Z) + \mu_f(A \wedge Z^c) \geq \rho_{(\lambda,f)}(A).
    \] 
    It follows that $\rho_{(\lambda,f)}$ is non-negative and increasing, which verifies \ref{r2}.
    
    For $X\in\mathcal{L}$, let $a_1, a_2 \in \mathcal{A}(\mathcal{L}) \backslash \mathcal{A}(X)$ such that $\len([X,X\dot\vee a_1\dot\vee a_2])=2$. We will show that $\rho_{(\lambda,f)}$ is submodular on $[X,X\dot\vee a_1\dot\vee a_2]$, which, by Corollary~\ref{cor:int2}, will imply the theorem. By \ref{z1}\ref{z1i}, for $k=1,2$, there exists $Z_k \in \mathcal{Z}(X \dot\vee a_k)$ such that  
    \[
        \rho_{(\lambda,f)}(X \dot\vee a_k) = \lambda(Z_k) + \mu_f((X \dot\vee a_k) \wedge Z_k^c),
    \] 
    for any $Z_k^c \in \mathbf{C}(Z_k; X \dot\vee a_k)$.  
    Therefore, choosing such $Z_1$ and $Z_2$, and applying \ref{z2}, we obtain  
    \begin{align}
        &\rho_{(\lambda,f)}(X \dot\vee a_1) + \rho_{(\lambda,f)}(X \dot\vee a_2) \label{eq:submod_proof}\\
        &\quad= \lambda(Z_1) + \mu_f((X \dot\vee a_1) \wedge Z_1^c) + \lambda(Z_2) + \mu_f((X \dot\vee a_2) \wedge Z_2^c) \nonumber \\
        &\quad\geq \lambda(Z_1 \wedge_{\mathcal{Z}} Z_2) + \lambda(Z_1 \vee_{\mathcal{Z}} Z_2) + \mu_f(X \wedge Z_1 \wedge Z_2 \wedge (Z_1 \wedge_{\mathcal{Z}} Z_2)^c) \nonumber \\
        &\quad\hspace{0.3cm} + \mu_f((X \dot\vee a_1) \wedge Z_1^c) + \mu_f((X \dot\vee a_2) \wedge Z_2^c) \nonumber
    \end{align}
    for any $(Z_1\wedge_\mathcal{Z}Z_2)^c\in\mathbf{C}(Z_1\wedge_\mathcal{Z}Z_2)$. By the inequality in Theorem~\ref{lem:q-poly_mu_ineq}, we thus get that $\rho_{(\lambda,f)}(X \dot\vee a_1) + \rho_{(\lambda,f)}(X \dot\vee a_2)$ is bounded below by
    \begin{align}
        \lambda(Z_1 \wedge_{\mathcal{Z}} Z_2) + \lambda(Z_1 \vee_{\mathcal{Z}} Z_2) + \mu_f(X \wedge Z_1 \wedge Z_2 \wedge (Z_1 \wedge_{\mathcal{Z}} Z_2)^c)&+\mu_f(X \wedge (Z_1 \wedge Z_2)^c) \label{eq:bigZZ_terms}\\
        &+ \mu_f((X \dot\vee a_1 \dot\vee a_2) \wedge (Z_1 \vee_{\mathcal{Z}} Z_2)^c)\nonumber
    \end{align}
    for some $(Z_1 \vee_{\mathcal{Z}} Z_2)^c \in \textup{\textbf{C}}(Z_1 \vee_{\mathcal{Z}} Z_2; X \dot\vee a_1 \dot\vee a_2)$ and some $(Z_1 \wedge Z_2)^c \in \textup{\textbf{C}}(Z_1 \wedge Z_2; X)$.
    Let $(Z_1\wedge_\mathcal{Z}Z_2)^{\hat{c}}\in\mathbf{C}(Z_1\wedge_\mathcal{Z}Z_2; X\wedge Z_1\wedge Z_2)$.
    By the modular law, if $C\leq A$ and $A\wedge B=\zero$, then $A\wedge(B\vee C)=C$. Setting $A=Z_1\wedge Z_2$, $B=(Z_1\wedge Z_2)^c$, and $C=Z_1\wedge Z_2\wedge(Z_1\wedge_\mathcal{Z}Z_2)^{\hat{c}}$ gives
    $$Z_1\wedge Z_2\wedge((Z_1\wedge Z_2)^c\dot\vee(Z_1\wedge Z_2\wedge (Z_1\wedge_\mathcal{Z}Z_2)^{\hat{c}}))=Z_1\wedge Z_2\wedge(Z_1\wedge_\mathcal{Z}Z_2)^{\hat{c}}.$$
    Since $Z_1\wedge_\mathcal{Z}Z_2\leq Z_1\wedge Z_2$, it follows that
    $$Z_1\wedge_\mathcal{Z} Z_2\wedge((Z_1\wedge Z_2)^c\dot\vee(Z_1\wedge Z_2\wedge (Z_1\wedge_\mathcal{Z}Z_2)^{\hat{c}}))=\zero.$$

    Recall that our choice of $(Z_1\wedge_\mathcal{Z}Z_2)^c\in\mathbf{C}(Z_1\wedge_\mathcal{Z}Z_2)$, which came from \ref{z2}, was arbitrary. By Corollary~\ref{cor:extend_to_complement}, we extend $(Z_1\wedge Z_2)^c\dot\vee(Z_1\wedge Z_2\wedge (Z_1\wedge_\mathcal{Z}Z_2)^{\hat{c}})$ to a complement $(Z_1\wedge_\mathcal{Z}Z_2)^c\in\mathbf{C}(Z_1\wedge_\mathcal{Z}Z_2)$ (i.e., we are now fixing $(Z_1\wedge_\mathcal{Z}Z_2)^c$ with this specification). Explicitly, we have
    \begin{equation}\label{eq:explic_ineq}
        (Z_1\wedge Z_2)^c\dot\vee(Z_1\wedge Z_2\wedge (Z_1\wedge_\mathcal{Z}Z_2)^{\hat{c}})\leq (Z_1\wedge_\mathcal{Z}Z_2)^c.
    \end{equation}
    Recall that $(Z_1\wedge_\mathcal{Z}Z_2)^{\hat{c}}\in\mathbf{C}(Z_1\wedge_\mathcal{Z}Z_2; X\wedge Z_1\wedge Z_2)$. We then have the following:
    \begin{align}
        X&=(X\wedge Z_1\wedge Z_2)\dot\vee(X\wedge (Z_1\wedge Z_2)^c)\nonumber\\
        &=(X\wedge (Z_1\wedge_\mathcal{Z}Z_2))\dot\vee(X\wedge Z_1\wedge Z_2\wedge(Z_1\wedge_\mathcal{Z}Z_2)^{\hat{c}})\dot\vee(X\wedge(Z_1\wedge Z_2)^c)\nonumber\\
        &\leq (X\wedge (Z_1\wedge_\mathcal{Z}Z_2))\dot\vee (X\wedge(( Z_1\wedge Z_2\wedge(Z_1\wedge_\mathcal{Z}Z_2)^{\hat{c}})\dot\vee(Z_1\wedge Z_2)^c))\quad\textup{(by Lemma~\ref{lem:gen_latt_distrib})}\nonumber\\
        &\leq (X\wedge(Z_1\wedge_\mathcal{Z}Z_2))\dot\vee(X\wedge(Z_1\wedge_\mathcal{Z}Z_2)^c)\quad\textup{(by }(\ref{eq:explic_ineq}))\nonumber\\
        &\leq X.\nonumber
    \end{align}
    Therefore, we have shown that  $(Z_1\wedge_\mathcal{Z}Z_2)^c\in\mathbf{C}(Z_1\wedge_\mathcal{Z}Z_2;X)$.
    
    Recall that $(Z_1\wedge Z_2)^c\leq (Z_1\wedge_\mathcal{Z}Z_2)^c$. Since $(Z_1\wedge Z_2)^c\in\mathbf{C}(Z_1\wedge Z_2;X)$, and by the modular law on $\mathcal{L}$, we have
    \begin{align*}
        X\wedge(Z_1\wedge_\mathcal{Z}Z_2)^c&=((X\wedge Z_1\wedge Z_2)\dot\vee(X\wedge(Z_1\wedge Z_2)^c)\wedge (Z_1\wedge_\mathcal{Z}Z_2)^c\\
        &=(X\wedge Z_1\wedge Z_2\wedge(Z_1\wedge_\mathcal{Z}Z_2)^c)\dot\vee(X\wedge(Z_1\wedge Z_2)^c).
    \end{align*}
    By Lemma~\ref{lem:mu_submod}, we therefore obtain
    \[
        \mu_f(X \wedge (Z_1 \wedge_{\mathcal{Z}} Z_2)^c) - \mu_f(X \wedge (Z_1 \wedge Z_2)^c) \leq \mu_f(X \wedge Z_1 \wedge Z_2 \wedge (Z_1 \wedge_{\mathcal{Z}} Z_2)^c).
    \] 
    Substituting this into (\ref{eq:bigZZ_terms}), we deduce from (\ref{eq:submod_proof}) that 
    \begin{align*}
        \rho_{(\lambda,f)}(X \dot\vee a_1) + \rho_{(\lambda,f)}(X \dot\vee a_2) 
        &\geq \lambda(Z_1 \wedge_{\mathcal{Z}} Z_2) + \mu_f(X \wedge (Z_1 \wedge_{\mathcal{Z}} Z_2)^c) \\
        &\hspace{0.3cm} + \lambda(Z_1 \vee_{\mathcal{Z}} Z_2) + \mu_f((X \dot\vee a_1 \dot\vee a_2) \wedge (Z_1 \vee_{\mathcal{Z}} Z_2)^c) \\
        &\geq \rho_{(\lambda,f)}(X) + \rho_{(\lambda,f)}(X \dot\vee a_1 \dot\vee a_2),
    \end{align*}
    where the final inequality comes from the definition of $\rho_{(\lambda,f)}$.
    
    We have thus established that $\rho_{(\lambda,f)}$ satisfies \ref{r3} on $[X,X\dot\vee a_1\dot\vee a_2]$. It now follows by Corollary~\ref{cor:int2} that $(\mathcal{L}, \rho_{(\lambda,f)})$ is an $\mathcal{L}$-polymatroid.
\end{proof}

We have established that if axioms \ref{z1}, \ref{z2}, and \ref{z5} hold for 
$(\mZ,\lambda,f)$, then $\rho_{(\lambda,f)}$ is the rank function of an $\mL$-polymatroid. In the remainder of this section, we will show that if $(\mZ,\lambda,f)$ furthermore satisfies the remaining cyclic flat axioms, then $\mZ$ indeed coincides with the lattice of cyclic flats of the $\mL$-polymatroid $(\mL,\rho_{(\lambda,f)})$. This will complete the cryptomorphism of Theorem~\ref{th:main}. We first require further technical results.

\begin{lemma}\label{lem:Z_ineq_switch}
    Suppose that $(\mZ,\lambda,f)$ satisfies axioms \ref{z1}, \ref{z2}, and \ref{z3}. Let $Z_1,Z_2\in\mathcal{Z}$. Then for all $Z_2^c\in\textup{\textbf{C}}(Z_2;Z_1)$, we have $\lambda(Z_1\vee_\mathcal{Z} Z_2)\leq\lambda(Z_2)+\mu_f(Z_1\wedge Z_2^c).$
\end{lemma}

\begin{proof}
    If $Z_1\leq Z_2$, the statement follows trivially. If $Z_2\leq Z_1$, then by \ref{z3}, for all $Z_2^c\in\mathbf{C}(Z_2)$, we have
    \[
    \lambda(Z_1\vee_\mathcal{Z} Z_2) = \lambda(Z_1) \leq \lambda(Z_2) + \mu_f(Z_1\wedge Z_2^c).
    \]
    
    Now assume that $Z_1\nleq Z_2$ and $Z_2\nleq Z_1$.
    Applying \ref{z3} to $Z_1$ and $Z_1 \wedge_\mZ Z_2$, we obtain
    \begin{align*}
        \lambda(Z_1)&\leq\lambda(Z_1\wedge_\mathcal{Z} Z_2)+\mu_f(Z_1\wedge(Z_1\wedge_\mathcal{Z} Z_2)^c),
    \end{align*}
    for any $(Z_1\wedge_\mathcal{Z} Z_2)^c\in\mathbf{C}(Z_1\wedge_\mathcal{Z}Z_2)$.
    From \ref{z2}, we have
    \begin{align*}
        \lambda(Z_1\vee_\mathcal{Z} Z_2)+\lambda(Z_1\wedge_\mathcal{Z} Z_2) &\leq \lambda(Z_1)+\lambda(Z_2)-\mu_f(Z_1\wedge Z_2\wedge (Z_1\wedge_\mathcal{Z} Z_2)^c),
    \end{align*}
    for any $(Z_1\wedge_\mathcal{Z} Z_2)^c\in\mathbf{C}(Z_1\wedge_\mathcal{Z}Z_2)$.
    Adding these two inequalities yields
    \begin{equation}\label{eq:33}
         \lambda(Z_1\vee_\mathcal{Z} Z_2) \leq \lambda(Z_2) + \mu_f(Z_1\wedge(Z_1\wedge_\mathcal{Z} Z_2)^c) - \mu_f(Z_1\wedge Z_2\wedge (Z_1\wedge_\mathcal{Z} Z_2)^c).
    \end{equation}

    Note that our choice of $(Z_1\wedge_\mathcal{Z}Z_2)^c\in\cC(Z_1 \wedge_\mZ Z_2)$ in (\ref{eq:33}) is arbitrary.
    We will now show that for any $Z_2^c \in \cC(Z_2;Z_1)$, there exists $(Z_1\wedge_\mathcal{Z}Z_2)^c\in\mathbf{C}(Z_1\wedge_\mathcal{Z}Z_2)$ such that 
    \begin{equation}\label{eq:late_mu_ineq}
        \mu_f(Z_1\wedge(Z_1\wedge_\mathcal{Z} Z_2)^c) \leq \mu_f(Z_1\wedge Z_2\wedge (Z_1\wedge_\mathcal{Z} Z_2)^c) + \mu_f(Z_1\wedge Z_2^c),
    \end{equation}
    from which we will deduce the result.

    Let $Z_2^c \in \cC(Z_2;Z_1)$. By Corollary~\ref{cor:extend_to_complement}, there exists $(Z_1 \wedge_\mZ Z_2)^c\in\mathbf{C}(Z_1\wedge_\mathcal{Z}Z_2)$ such that $Z_2^c\leq(Z_1\wedge_\mathcal{Z}Z_2)^c$.
    Since 
    \[
    Z_1\wedge Z_2\wedge (Z_1\wedge_\mathcal{Z} Z_2)^c \leq Z_1\wedge(Z_1\wedge_\mathcal{Z} Z_2)^c \quad \text{and} \quad Z_1\wedge Z_2^c \leq Z_1\wedge(Z_1\wedge_\mathcal{Z} Z_2)^c,
    \]
    we have that 
    \begin{equation}\label{ineq:z1z2Zcleq}
        (Z_1\wedge Z_2\wedge(Z_1\wedge_\mathcal{Z} Z_2)^c )\vee (Z_1\wedge Z_2^c) \leq Z_1\wedge(Z_1\wedge_\mathcal{Z} Z_2)^c.
    \end{equation}
    Since $Z_1\wedge_\mathcal{Z}Z_2\leq Z_1\wedge Z_2$, Proposition~\ref{prop:compdecomp} implies that $(Z_1\wedge_\mathcal{Z}Z_2)^c\in\mathbf{C}(Z_1\wedge_\mathcal{Z}Z_2;Z_1\wedge Z_2)$. Therefore, using the modularity of $\mathcal{L}$, we compute
    \begin{align*}
        \operatorname{h}((Z_1\wedge Z_2\wedge(Z_1\wedge_\mathcal{Z} Z_2)^c )\dot\vee (Z_1\wedge Z_2^c))&=\operatorname{h}(Z_1\wedge Z_2\wedge (Z_1\wedge_\mathcal{Z} Z_2)^c)+\operatorname{h}(Z_1\wedge Z_2^c)\\
        &= \operatorname{h}(Z_1\wedge Z_2) - \operatorname{h}(Z_1\wedge_\mathcal{Z} Z_2) + \operatorname{h}(Z_1) - \operatorname{h}(Z_1\wedge Z_2) \\
        &= \operatorname{h}(Z_1) - \operatorname{h}(Z_1\wedge_\mathcal{Z} Z_2) \\
        &= \operatorname{h}(Z_1\wedge(Z_1\wedge_\mathcal{Z} Z_2)^c),
    \end{align*}
    which by (\ref{ineq:z1z2Zcleq}) confirms that 
    \[
    Z_1\wedge(Z_1\wedge_\mathcal{Z} Z_2)^c = (Z_1\wedge Z_2\wedge (Z_1\wedge_\mathcal{Z} Z_2)^c)\dot\vee(Z_1\wedge Z_2^c).
    \]
    By Lemma~\ref{lem:mu_submod}, we have thus shown that (\ref{eq:late_mu_ineq}) holds. By substituting (\ref{eq:late_mu_ineq}) into (\ref{eq:33}), we get
    $$\lambda(Z_1\vee_\mathcal{Z}Z_2)\leq\lambda(Z_2)+\mu_f(Z_1\wedge Z_2^c),$$
    where our choice of $Z_2^c\in\mathbf{C}(Z_2;Z_1)$ was arbitrary. This is the desired result.
\end{proof}

\begin{corollary}\label{cor:lambdaZ_is_rhoZ}
    Suppose that $(\mZ,\lambda,f)$ satisfies axioms \ref{z1}, \ref{z2}, and \ref{z3}. If $Z\in\mathcal{Z}$, then $\rho_{(\lambda,f)}(Z) = \lambda(Z)$.
\end{corollary}

\begin{proof}
    By definition, we have $\rho_{(\lambda,f)}(Z) \leq \lambda(Z)$. Conversely, for any $\Tilde{Z} \in \mathcal{Z}$, Lemma~\ref{lem:Z_ineq_switch} gives
    \[
    \lambda(Z) \leq \lambda(Z \vee_\mathcal{Z} \Tilde{Z}) \leq \lambda(\Tilde{Z}) + \mu_f(Z \wedge \Tilde{Z}^c)
    \]
    for any $\Tilde{Z}^c \in \textup{\textbf{C}}(\Tilde{Z};Z)$. Therefore, the definition of $\rho_{(\lambda,f)}$ gives $\lambda(Z)\leq\rho_{(\lambda,f)}(Z)$.
\end{proof}

\begin{lemma}\label{lem:hyper_not_cyclic}
    Suppose that $(\mZ,\lambda,f)$ satisfies axioms \ref{z1} and \ref{z6}. Let $X \in \mathcal{L}$. 
    Let $Z \in \mathcal{Z}(X)$ 
    satisfy \ref{z1} for $X$. If $Z < X$, then there exists $H \in \mH(X)$ such that for all $H^c \in \textup{\textbf{C}}(H)$, we have
    \[
    \rho_{(\lambda,f)}(X) - \rho_{(\lambda,f)}(H) = \mu_f(X \wedge H^c).
    \]
\end{lemma}

\begin{proof}
    Let $Z^c$ be any element of the set $\mathbf{C}(Z;H)\cap\mathbf{C}(Z;X)$, which by Corollary~\ref{cor:decomp_both}, is non-empty. By \ref{z1}\ref{z1ii}, $\mu_f$ is weakly decomposable on $[\zero,X \wedge Z^c]$. Using Lemma~\ref{lem:q-mod_implies_mu-support}, we can choose a maximal chain 
    \[
    \zero = H_m \lessdot \dots \lessdot H_0 = X\wedge Z^c
    \]
    (where $m=\textup{h}(X\wedge Z^c)$), with associated layering $L_1, \dots, L_m$ such that $f$ is constant on each layer $L_k$, for $k \in [m]$. Hence,  by Corollary~\ref{cor:mu_q-mod_gives_additive}, we deduce that
    \[
    \mu_f(X\wedge Z^c) = \sum_{i=1}^m f(a_i) \quad \text{and}\quad \mu_f(H_1) = \sum_{i=2}^m f(a_i)
    \]
    for arbitrary $a_i\in L_i\cap\mathcal{A}(X\wedge Z^c)$.
    Therefore, we obtain
    \begin{equation}\label{eq:hyper+atom}
        \mu_f(X\wedge Z^c) = \mu_f(H_1) + \mu_f(a) \quad \text{for any } a\in L_1.
    \end{equation}

    By the modularity of $\mathcal{L}$, we can choose $H \in \mathcal{H}(X)$ such that $H\wedge Z^c = H_1$. By \ref{z1}\ref{z1i}, there exists $\Tilde{Z} \in \mathcal{Z}$ such that
    \[
    \rho_{(\lambda,f)}(H) = \lambda(\Tilde{Z}) + \mu_f(H\wedge \Tilde{Z}^c)\quad  \textup{for all } \Tilde{Z}^c\in\textup{\textbf{C}}(\Tilde{Z};H).
    \]
    We will let $\tilde{Z}^c$ be any element of the set $\mathbf{C}(\tilde{Z};H)\cap\mathbf{C}(\tilde{Z};X)$, which by Corollary~\ref{cor:decomp_both}, is non-empty. Recall that $Z^c\in\mathbf{C}(Z;H)\cap\mathbf{C}(Z;X)$. By the definition of $\rho_{(\lambda,f)}$, we get
    \begin{align}
        \rho_{(\lambda,f)}(X) &= \lambda(Z) + \mu_f(X\wedge Z^c) \leq \lambda(\Tilde{Z}) + \mu_f(X\wedge\Tilde{Z}^c),\label{eqs:rho_X}\\
        \rho_{(\lambda,f)}(H) &= \lambda(\Tilde{Z}) + \mu_f(H\wedge\Tilde{Z}^c) \leq \lambda(Z) + \mu_f(H\wedge Z^c).\label{eqs:rho_hyper}
    \end{align}
    Adding these inequalities yields
    \begin{equation}\label{eq:hyper_mu_sum}
        \mu_f(X\wedge Z^c) + \mu_f(H\wedge \Tilde{Z}^c) \leq \mu_f(X\wedge\Tilde{Z}^c) + \mu_f(H\wedge Z^c),
    \end{equation}
    for any $\Tilde{Z}^c\in\textup{\textbf{C}}(\Tilde{Z};H)$.
    
    Since $a\in L_1$, we have that $a \nleq H_1$ and $a\leq X \wedge Z^c$. Since $H_1=H\wedge Z^c=H\wedge X\wedge Z^c$, this implies that $a\nleq H$, which means that $H\dot\vee a=X$.
    We consider two cases:

    \textbf{Case 1} \( a \leq \Tilde{Z} \): By the modularity of $\mathcal{L}$, we have in this case that $(H\dot\vee a)\wedge \tilde{Z}=(H\wedge \tilde{Z})\dot\vee a$.
    Since $\tilde{Z}^c\in\mathbf{C}(\tilde{Z};H)$, we have
    \begin{align*}
        \textup{h}(H\dot\vee a)&\geq\textup{h}((H\dot\vee a)\wedge\tilde{Z})+\textup{h}((H\dot\vee a)\wedge\tilde{Z}^c)\\
        &=1+\textup{h}(H\wedge\tilde{Z})+\textup{h}((H\dot\vee a)\wedge\tilde{Z}^c)\\
        &\geq 1+\textup{h}(H\wedge\tilde{Z})+\textup{h}(H\wedge\tilde{Z}^c)\\
        &=1+\textup{h}(H)\\
        &=\textup{h}(H\dot\vee a),
    \end{align*}
    from which we deduce that $(H\dot\vee a)\wedge \tilde{Z}^c=H\wedge \tilde{Z}^c$, and so since $X=H\dot\vee a$, we have $X\wedge\Tilde{Z}^c = H\wedge\Tilde{Z}^c.$
      By (\ref{eq:hyper_mu_sum}), this yields
      $\mu_f(X\wedge Z^c)\leq\mu_f(H\wedge Z^c).$
      However, we established in (\ref{eq:hyper+atom}) that $\mu_f(X\wedge Z^c)=\mu_f(H\wedge Z^c)+f(a)$, which would then imply that $f(a)=0$. Since $a\nleq Z$, we have $a\nleq0_\mathcal{Z}$. By \ref{z6}, we have that $f(a)>0$, which gives a contradiction. Therefore, $a\leq\tilde{Z}$ is not possible.

    \textbf{Case 2} \( a \not\leq \Tilde{Z} \): Since $a$ is an atom, we have $\tilde{Z}\wedge a=\zero$.
      Then, by Corollary~\ref{cor:extend_to_complement}, there exists $\Tilde{Z}^c\in\textup{\textbf{C}}(\Tilde{Z};H)$ such that $a\leq \Tilde{Z}^c$, implying, by the modularity of $\mathcal{L}$, that $X\wedge\Tilde{Z}^c = (H\wedge\Tilde{Z}^c) \dot\vee a.$
      Recall that $H_1=H\wedge Z^c$. We observe the following:
      \begin{align}\label{eq:mu_inequalities}
          \mu_f(H\wedge Z^c) + \mu_f(a) + \mu_f(H\wedge\Tilde{Z}^c) &= \mu_f(X\wedge Z^c) + \mu_f(H\wedge\Tilde{Z}^c)\quad(\textup{by } (\ref{eq:hyper+atom}))\nonumber \\
          &\leq \mu_f(X\wedge\Tilde{Z}^c) + \mu_f(H\wedge Z^c)\quad\textup{(by (\ref{eq:hyper_mu_sum}))} \nonumber\\
          &\leq \mu_f(H\wedge\Tilde{Z}^c) + \mu_f(a) + \mu_f(H\wedge Z^c)\quad\textup{(by Lemma \ref{lem:mu_submod})}. \nonumber
      \end{align}
      Thus, equality holds in (\ref{eq:hyper_mu_sum}). Therefore, equality must hold in in (\ref{eqs:rho_X}) and (\ref{eqs:rho_hyper}). From these equalities, we deduce that $\rho_{(\lambda,f)}(X) - \rho_{(\lambda,f)}(H) = \mu_f(X\wedge Z^c) - \mu_f(H\wedge Z^c).$
    Since $H\wedge Z^c=H_1 \lessdot X\wedge Z^c$, it follows that $\mu_f(X\wedge Z^c) - \mu_f(H\wedge Z^c) = \mu_f(a) = \mu_f(X\wedge H^c)$
    for some $H^c\in\textup{\textbf{C}}(H)$. As $a\in L_1$ was arbitrary, $H^c\in\textup{\textbf{C}}(H)$ is also arbitrary. The result follows.
\end{proof}

We implicitly use the result of Theorem~\ref{thm:axioms_imply_q-poly} for several of the following results.

\begin{lemma}\label{lem:Z_leq_flat}
    Suppose that $(\mZ,\lambda,f)$ satisfies axioms \ref{z1}, \ref{z2}, and \ref{z5}. Then, for any flat $F$ of the $\mathcal{L}$-polymatroid $(\mathcal{L},\rho_{(\lambda,f)})$, we have $Z \leq F$ for all $Z \in \mathcal{Z}(F)$.
\end{lemma}

\begin{proof}
    Let $F \in \mathcal{L}$ be a flat, and suppose that
    \begin{equation}\label{eq:rho_FZ}
        \rho_{(\lambda,f)}(F) = \lambda(Z) + \mu_f(F \wedge Z^c)
    \end{equation}
    for some $Z \in \mathcal{Z}$ and some $Z^c \in \textbf{C}(Z;F)$. Note that $F \wedge Z^c \leq (F \vee Z) \wedge Z^c$. Since $Z^c\in\mathbf{C}(Z;F)$, we have, using the modularity of $\mathcal{L}$,
    $$\textup{h}(F\wedge Z^c)=\textup{h}(F)-\textup{h}(F\wedge Z).$$
    Moreover, by Proposition~\ref{prop:compdecomp}, we have that $Z^c\in\mathbf{C}(Z;F\vee Z)$, from which we obtain
    $$\textup{h}((F\vee Z)\wedge Z^c)=\textup{h}(F\vee Z)-\textup{h}((F\vee Z)\wedge Z)=\textup{h}(F\vee Z)-\textup{h}(Z).$$
    Again by the modularity of $\mathcal{L}$, we have
    $\textup{h}(F)-\textup{h}(F\wedge Z)=\textup{h}(F\vee Z)-\textup{h}(Z),$
    from which we obtain $\textup{h}(F\wedge Z^c)=\textup{h}((F\vee Z)\wedge Z^c)$, and thus 
    $F \wedge Z^c = (F \vee Z) \wedge Z^c.$
    Substituting this into (\ref{eq:rho_FZ}), we get $\rho_{(\lambda,f)}(F) = \lambda(Z) + \mu_f((F \vee Z) \wedge Z^c).$
    Since $\rho_{(\lambda,f)}$ satisfies \ref{r2}, we also have $\rho_{(\lambda,f)}(F \vee Z) \geq \rho_{(\lambda,f)}(F).$
    Combining these inequalities, we obtain
    \[
    \rho_{(\lambda,f)}(F \vee Z) \geq \rho_{(\lambda,f)}(F) = \lambda(Z) + \mu_f((F \vee Z) \wedge Z^c) \geq \rho_{(\lambda,f)}(F \vee Z).
    \]
    This forces equality, meaning that $F \vee Z = F$, since $F$ is a flat. The result follows.
\end{proof}

\begin{corollary}\label{lem:cyc_flats_in_Z}
    Suppose that $(\mZ,\lambda,f)$ satisfies axioms \ref{z1}, \ref{z2}, \ref{z5}, and \ref{z6}. Every cyclic flat of the $\mathcal{L}$-polymatroid $(\mathcal{L},\rho_{(\lambda,f)})$ is an element of $\mathcal{Z}$.
\end{corollary}
\begin{proof}
    Let $F\in\mathcal{L}$ be a flat of $(\mathcal{L},\rho_{(\lambda,f)})$ and let $Z\in\mathcal{Z}(F)$. By Lemma~\ref{lem:Z_leq_flat}, we have that $Z \leq F$. If $F \notin \mathcal{Z}$, then $Z < F$. By Lemma~\ref{lem:hyper_not_cyclic}, there exists $H\in\mathcal{H}(F)$ such that
    $$\rho_{(\lambda,f)}(F)-\rho_{(\lambda,f)}(H)=\mu_f(X\wedge H^c)$$
    for all $H^c\in\mathbf{C}(H)$. This means that, by definition, $F$ is not cyclic. Therefore, we conclude that if $F\in\mathcal{L}$ is a cyclic flat of $(\mathcal{L},\rho_{(\lambda,f)})$, then $F\in\mathcal{Z}$.
\end{proof}

\begin{lemma}\label{lem:Z_are_cyc_flats}
    If $(\mZ,\lambda,f)$ satisfies axioms \ref{z1}--\ref{z6}, then every element of $\mathcal{Z}$ is a cyclic flat of the $\mathcal{L}$-polymatroid $(\mathcal{L},\rho_{(\lambda,f)})$.
\end{lemma}

\begin{proof}
    Let $Z\in\mathcal{Z}$. First, we show that $Z$ is a flat. Let $a\in\mathcal{L}$ be an atom not contained in $Z$. By Corollary~\ref{cor:lambdaZ_is_rhoZ}, we have $\rho_{(\lambda,f)}(Z)=\lambda(Z)$. By \ref{z1}, we may choose $\tilde{Z}\in\mZ(Z)$ such that 
    \[
    \rho_{(\lambda,f)}(Z\dot\vee a)=\lambda(\tilde{Z})+\mu_f\bigl((Z\dot\vee a)\wedge\tilde{Z}^c\bigr)
    \]
    for all $\tilde{Z}^c\in\cC(\tilde{Z};Z\dot\vee a)$.

    If $Z=\tilde{Z}$, then, as $a\nleq Z$ and $\tilde{Z}^c\in\mathbf{C}(\tilde{Z};Z\dot\vee a)$, we must have $(Z\dot\vee a)\wedge \tilde{Z}^c\neq\zero$, which meets (in $\mathcal{L}$) trivially with $0_\mathcal{Z}$ since $0_\mathcal{Z}\leq\tilde{Z}\in\mathcal{Z}$. Therefore, by \ref{z6}, we have $\mu_f((Z\dot\vee a)\wedge \tilde{Z}^c)>0$, from which we obtain $\rho_{(\lambda,f)}(Z)<\rho_{(\lambda,f)}(Z\dot\vee a)$.

    If $Z<\tilde{Z}$, then by \ref{z3*} we have $\lambda(Z)<\lambda(\tilde{Z})$, from which we obtain $\rho_{(\lambda,f)}(Z)<\rho_{(\lambda,f)}(Z\dot\vee a)$.
    Hence, we now suppose that $Z\nleq\tilde{Z}$. Recall that $\rho_{(\lambda,f)}(Z)=\lambda(Z)$. We observe the following:
    \begin{align}
        \lambda(Z)&\leq\lambda(Z\vee_\mathcal{Z}\tilde{Z})\quad \text{(by \ref{z3})}\label{ineq:Z_tildeZ_join}\\
        &\leq\lambda(\tilde{Z})+\mu_f\bigl(Z\wedge\tilde{Z}^c\bigr)\quad\text{(by Lemma~\ref{lem:Z_ineq_switch})}\nonumber\\
        &\leq\lambda(\tilde{Z})+\mu_f\bigl((Z\dot\vee a)\wedge\tilde{Z}^c\bigr)\label{ineq:tildeZ_q-mod}\\
        &=\rho_{(\lambda,f)}(Z\dot\vee a),\nonumber
    \end{align}
    where (\ref{ineq:tildeZ_q-mod}) holds since, by \ref{z1}\ref{z1ii} and Proposition~\ref{cor:mu_is_increasing}, $\mu_f$ is increasing on 
    \(
    [\zero,(Z\dot\vee a)\wedge\tilde{Z}^c].
    \)
    If $\tilde{Z}\nleq Z$, then $Z<Z\vee_\mathcal{Z}\tilde{Z}$, which means that (\ref{ineq:Z_tildeZ_join}) is a strict inequality by \ref{z3*}. 
    
    If $\tilde{Z}\leq Z$, then since $((Z\dot\vee a)\wedge\tilde{Z})\dot\vee(Z\wedge\tilde{Z}^c)\leq Z<Z\dot\vee a$, yet $\tilde{Z}^c\in\mathbf{C}(\tilde{Z};Z\dot\vee a)$, we must have $Z\wedge \tilde{Z}^c<(Z\dot\vee a)\wedge\tilde{Z}^c$, so (\ref{ineq:tildeZ_q-mod}) is a strict inequality by \ref{z1}\ref{z1ii} and \ref{z6}. It follows that $Z$ is a flat of $(\mathcal{L},\rho_{(\lambda,f)})$.
     
    We now show that $Z$ is cyclic. Let $H\in\mH(Z)$ and by \ref{z1} choose $\hat{Z}\in\mathcal{Z}(H)$ such that
    \[
    \rho_{(\lambda,f)}(H)=\lambda(\hat{Z})+\mu_f\bigl(H\wedge\hat{Z}^c\bigr)
    \]
    for all $\hat{Z}^c\in\cC(\hat{Z};H)$. If $\rho_{(\lambda,f)}(Z)=\rho_{(\lambda,f)}(H)$, then there is nothing to show, so assume that $\rho_{(\lambda,f)}(Z)>\rho_{(\lambda,f)}(H)$. Since $(\mathcal{L},\rho_{(\lambda,f)})$ is an $\mL$-polymatroid and $Z$ is a flat, we can deduce that $H$ is also a flat, because for all $a\in\mathcal{A}(\mathcal{L})\backslash\mathcal{A}(Z)$, the submodularity of $\rho_{(\lambda,f)}$ gives
    \[
    0<\rho_{(\lambda,f)}(Z\vee a)-\rho_{(\lambda,f)}(Z)\leq \rho_{(\lambda,f)}(H\vee a)-\rho_{(\lambda,f)}(H).
    \]
    Then, by Lemma~\ref{lem:Z_leq_flat}, we have that $\hat{Z}\leq H$ since $\hat{Z}\in\mathcal{Z}(H)$ and $H$ is a flat.

    Note that $\hat{Z}\leq H\lessdot Z$, which means that $H\in\mathcal{H}(Z)\cap[\hat{Z},Z]$. Therefore, axiom \ref{z3*} gives that for some $\hat{Z}^c\in\mathbf{C}(\hat{Z})$, we have
    \begin{equation}\label{eq:z4_Hproof}
        \lambda(Z)-\lambda(\hat{Z})<\mu_f(H\wedge \hat{Z}^c)+\mu_f(H^c\wedge Z)
    \end{equation}
    for some $H^c\in\mathbf{C}(H)$. By Corollary~\ref{cor:lambdaZ_is_rhoZ}, we have $\rho_{(\lambda,f)}(Z)=\lambda(Z)$. By Proposition~\ref{prop:compdecomp}, we have that $\mathbf{C}(\hat{Z})=\mathbf{C}(\hat{Z};H)$ since $\hat{Z}\leq H$. Therefore, (\ref{eq:z4_Hproof}) and the choice of $\hat{Z}$ give 
    $$\rho_{(\lambda,f)}(Z)=\lambda(Z)=\rho_{(\lambda,f)}(H)+\mu_f\bigl(H^c\wedge Z\bigr).$$
    This completes the proof that $Z$ is cyclic and so completes the proof of the lemma.
\end{proof}

The proof of Theorem~\ref{th:main} is now a consequence of the preceding results.

\begin{proof}[Proof of Theorem~\ref{th:main}]
    Since $(\mathcal{Z},\lambda,f)$ satisfies axioms \ref{z1}--\ref{z6}, then by Theorem~\ref{thm:axioms_imply_q-poly}, Corollary~\ref{lem:cyc_flats_in_Z}, and Lemma~\ref{lem:Z_are_cyc_flats}, we have that $(\mathcal{L},\rho_{(\lambda,f)})$ is an $\mL$-polymatroid with lattice of cyclic flats $\mathcal{Z}$.
\end{proof}

\begin{example}
    In Figure~\ref{fig:cyc_flat_example}, we present a simple example of a cover-weighted lattice $(\mathcal{Z},\lambda)$, where $\lambda(a)=3$ and $\lambda(b)=\lambda(c)=\lambda(d)=4$. We assign weights 
    \[
    f(x_1)=4,\quad f(x_2)=2,\quad f(x_3)=2,\quad f(x_4)=2,\quad f(x_5)=4,
    \]
    to the atoms in the ambient lattice $\mathcal{L}$ in which $\mathcal{Z}$ is embedded. Note that $\mathcal{L}$ is complemented and modular, but is neither a subspace lattice nor a Boolean lattice.

    \begin{figure}[h!]
        \begin{minipage}{0.45\textwidth}
            \begin{center}
            \scalebox{1.4}{
            \begin{tikzpicture}
                \node[] (0) at (0,0) {\tiny $\mathbf{0}$};
                \node[] (1) at (0,2) {\tiny $\mathbf{1}$};
                \node[] (i) at (-2,1) {\tiny $a$};
                \node[] (j) at (-1,1) {\tiny $b$};
                \node[] (l) at (1,1) {\tiny $c$};
                \node[] (m) at (2,1) {\tiny $d$};
                \path [-,cyan] (0) edge (i);
                \path [-,blue] (0) edge (j);
                \path [-,blue] (0) edge (l);
                \path [-,blue] (0) edge (m);
                \path [-,black] (1) edge (i);
                \path [-,red] (1) edge (j);
                \path [-,red] (1) edge (l);
                \path [-,red] (1) edge (m);
            \end{tikzpicture}}
            \end{center}
        \end{minipage}
        \begin{minipage}{0.5\textwidth}
            \begin{center}
            \scalebox{1.4}{
            \begin{tikzpicture}
                \node[] (0) at (0,0) {\tiny $\mathbf{0}$};
                \node[] (1) at (0,3) {\tiny $\mathbf{1}$};
                \node[] (b) at (-2,1) {\tiny $x_1$};
                \node[] (c) at (-1,1) {\tiny $x_2$};
                \node[] (d) at (0,1) {\tiny $x_3$};
                \node[] (e) at (1,1) {\tiny $x_4$};
                \node[] (f) at (2,1) {\tiny $x_5$};
                \node[] (i) at (-2,2) {\tiny $a$};
                \node[] (j) at (-1,2) {\tiny $b$};
                \node[] (k) at (0,2) {\tiny $c$};
                \node[] (l) at (1,2) {\tiny $d$};
                \node[] (m) at (2,2) {\tiny $e$};
                \path [-,cyan] (0) edge (b);
                \path [-,black] (0) edge (c);
                \path [-,black] (0) edge (d);
                \path [-,black] (0) edge (e);
                \path [-,cyan] (0) edge (f);
                \path [-,lightgray] (1) edge (i);
                \path [-,lightgray] (1) edge (j);
                \path [-,lightgray] (1) edge (k);
                \path [-,lightgray] (1) edge (l);
                \path [-,lightgray] (1) edge (m);
                \path [-,lightgray] (b) edge (j);
                \path [-,lightgray] (b) edge (k);
                \path [-,lightgray] (c) edge (i);
                \path [-,lightgray] (c) edge (k);
                \path [-,lightgray] (d) edge (i);
                \path [-,lightgray] (d) edge (j);
                \path [-,lightgray] (d) edge (l);
                \path [-,lightgray] (d) edge (m);
                \path [-,lightgray] (e) edge (k);
                \path [-,lightgray] (e) edge (m);
                \path [-,lightgray] (f) edge (k);
                \path [-,lightgray] (f) edge (l);
            \end{tikzpicture}}
            \end{center}
        \end{minipage}
        \caption{On the left is the weighted lattice $(\mathcal{Z},\lambda)$. On the right is the lattice $\mL$, whose atoms are weighted by $f$. The triple $(\mZ,\lambda,f)$ satisfies \ref{z1}--\ref{z6}. (\textcolor{red}{1}, 2, \textcolor{cyan}{3}, \textcolor{blue}{4})}
        \label{fig:cyc_flat_example}
    \end{figure}

    It is easily verified that the construction in Figure~\ref{fig:cyc_flat_example} satisfies the cyclic flat axioms. Therefore, we can construct a unique $\mL$-polymatroid $(\mathcal{L},\rho_{(\lambda,f)})$ such that $(\mathcal{Z},\lambda,f)$ is its weighted lattice of cyclic flats. The $\mL$-polymatroid $(\mathcal{L},\rho_{(\lambda,f)})$ is shown in Figure~\ref{fig:cyc_flat_example_final}. 

    \begin{figure}[h]
        \centering
        \scalebox{1.6}{
        \begin{tikzpicture}
            \node[] (0) at (0,0) {\tiny $\mathbf{0}$};
            \node[] (1) at (0,3) {\tiny $\mathbf{1}$};
            \node[] (b) at (-2,1) {\tiny $x_1$};
            \node[] (c) at (-1,1) {\tiny $x_2$};
            \node[] (d) at (0,1) {\tiny $x_3$};
            \node[] (e) at (1,1) {\tiny $x_4$};
            \node[] (f) at (2,1) {\tiny $x_5$};
            \node[] (i) at (-2,2) {\tiny $a$};
            \node[] (j) at (-1,2) {\tiny $b$};
            \node[] (k) at (0,2) {\tiny $c$};
            \node[] (l) at (1,2) {\tiny $d$};
            \node[] (m) at (2,2) {\tiny $e$};
            \path [-,cyan] (0) edge (b);
            \path [-,black] (0) edge (c);
            \path [-,black] (0) edge (d);
            \path [-,black] (0) edge (e);
            \path [-,cyan] (0) edge (f);
            \path [-,black] (1) edge (i);
            \path [-,red] (1) edge (j);
            \path [-,red] (1) edge (k);
            \path [-,red] (1) edge (l);
            \path [-,red] (1) edge (m);
            \path [-,red] (b) edge (j);
            \path [-,red] (b) edge (k);
            \path [-,red] (c) edge (i);
            \path [-,black] (c) edge (k);
            \path [-,red] (d) edge (i);
            \path [-,black] (d) edge (j);
            \path [-,black] (d) edge (l);
            \path [-,black] (d) edge (m);
            \path [-,black] (e) edge (k);
            \path [-,black] (e) edge (m);
            \path [-,red] (f) edge (k);
            \path [-,red] (f) edge (l);
        \end{tikzpicture}}
        \caption{A representation of the $\mathcal{L}$-polymatroid $(\mathcal{L},\rho_{(\lambda,f)})$ whose lattice of cyclic flats is $(\mathcal{Z},\lambda,f)$, as shown in Figure~\ref{fig:cyc_flat_example}. (\textcolor{red}{1}, 2, \textcolor{cyan}{3}, \textcolor{blue}{4})}
        \label{fig:cyc_flat_example_final}
    \end{figure}
\end{example}

\section{Final Comments}\label{sec:comments}

In this final section, we comment on the cyclic flat axioms given in Definition~\ref{def:axioms} and compare them with the existing cyclic flat axioms in the literature.

\begin{remark}
    We discuss the similarities and differences between axioms \ref{z1}--\ref{z6} presented here and those in \cite[Section 3.1]{Csirm20}, which apply to polymatroids.
    \begin{itemize}
        \item Axiom \ref{z1} is vacuously satisfied in the case that $\mL$ is the Boolean lattice, since in that case complements are unique. This immediately implies that \ref{z1}\ref{z1i} holds, while \ref{z1}\ref{z1ii} also holds because $\mu_f$ is a valuation on $\mL$.
        \item In the case that $\mL$ is the Boolean lattice, $\mu_f$ is increasing on $\mL$, and hence axiom \ref{z2} can be replaced by the following:
        \[
          \lambda(Z_1)+\lambda(Z_2)\geq\lambda(Z_1\wedge_\mathcal{Z} Z_2)+\lambda(Z_1\vee_\mathcal{Z} Z_2)+\mu_f\Bigl( Z_1\wedge Z_2\wedge\Bigl(Z_1\wedge_\mathcal{Z} Z_2\Bigr)^c\Bigr),
        \]
        for any $Z_1,Z_2 \in \mZ$ (complements are unique in a Boolean lattice).
        \item As complements are unique in a Boolean lattice, along with the function $\mu_f$ being a valuation on a Boolean lattice, we have that \ref{z3*} coincides with one of the cyclic flat axioms found in \cite{Csirm20}.
    \end{itemize}
\end{remark}

\begin{remark}
    We now comment on some similarities and differences between axioms \ref{z1}--\ref{z6} and those provided in \cite[Definition 3.1]{AlfByr22}, which apply to $q$-matroids. Recall from Remark~\ref{rem:q-mat_mu} that if $(\mathcal{L},r)$ is a $q$-matroid and $X\wedge\textup{cl}(\zero)=\zero,$
    then $\mu_r(X)=\textup{dim}(X)$. For any cyclic flat $Z\in\mathcal{L}$, we have $\textup{cl}(\zero)\leq Z$. This means that $\mu_r(X)=\textup{dim}(X)$ for any $X\leq Z^c$, for any $Z^c\in\textup{\textbf{C}}(Z)$. Since in axioms \ref{z1}--\ref{z6}, $\mu_f$ is applied only to lattice elements contained in complements of cyclic flats, the dimension function is used in \cite{AlfByr22} instead of $\mu_f$. For the remainder of this remark, we set $\mu_f(\cdot)=\dim(\cdot).$
    \begin{itemize}
        \item Let $Z\in\mathcal{Z}$. Since $\textup{dim}(Z^c)=\textup{dim}(\one)-\textup{dim}(Z)$
        for any $Z^c\in\textup{\textbf{C}}(Z)$, axiom \ref{z1}\ref{z1i} is satisfied. Moreover, since $\textup{dim}$ is a valuation, axiom \ref{z1}\ref{z1ii} is also satisfied.
        \item Since $\dim(\cdot)$ is an increasing function, axiom \ref{z2} becomes equivalent to \cite[(Z3)]{AlfByr22}.
        \item Since $\dim(\cdot)$ is a valuation, axiom \ref{z3*} becomes equivalent to \cite[(Z2)]{AlfByr22}.
        \item Since $\textup{dim}(e)=1>0$ for any $e\in\mA(\mathcal{L})$, axiom \ref{z6} is implicitly included in \cite{AlfByr22}.
    \end{itemize}
\end{remark}

To provide some intuition about the role of \ref{z1}, we now present examples showing that it is not a consequence of the remaining axioms.
\begin{example}   
    Consider the lattice $\mathcal{L}(\mathbb{F}_2^3)$ with the atomic weights as shown in Figure~\ref{fig:z1ii_counter}. Let 
    \[
    \mathcal{Z}_1=\{\langle e_1 \rangle\} \quad \text{and} \quad \mathcal{Z}_2 = \{\langle e_1 \rangle,\,\F_2^3\}.
    \]
    It is easy to see that $\mathcal{Z}_2$ satisfies all of the cyclic flat axioms, while $\mathcal{Z}_1$ satisfies all the axioms except for \ref{z1}\ref{z1ii}. This shows that \ref{z1}\ref{z1ii} is not implied by the remaining axioms.

    We remark further that by the cover-weight axioms \ref{cw1}--\ref{cw2}, one can deduce the remaining cover-weights solely from the atomic weights of $\mathcal{L}(\mathbb{F}_2^3)$. It can then be checked that the resulting $q$-polymatroid has $\mathcal{Z}_2$ as its set of cyclic flats.
    
    \begin{figure}[h]
        \centering
        \begin{tabular}{cc}
          \scalebox{1.1}{
\begin{tikzpicture}
            \node[] (0) at (0,0) {\tiny $0$};
            \node[] (1) at (0,3) {\tiny $\mathbb{F}_2^3$};
            \node[] (a) at (-3,1) {\tiny $100$};
            \node[] (b) at (-2,1) {\tiny $010$};
            \node[] (c) at (-1,1) {\tiny $110$};
            \node[] (d) at (0,1) {\tiny $111$};
            \node[] (e) at (1,1) {\tiny $011$};
            \node[] (f) at (2,1) {\tiny $001$};
            \node[] (g) at (3,1) {\tiny $101$};
            \node[] (h) at (-3,2) {\tiny $\begin{smallmatrix}100\\010\end{smallmatrix}$};
            \node[] (i) at (-2,2) {\tiny $\begin{smallmatrix}100\\011\end{smallmatrix}$};
            \node[] (j) at (-1,2) {\tiny $\begin{smallmatrix}100\\001\end{smallmatrix}$};
            \node[] (k) at (0,2) {\tiny $\begin{smallmatrix}010\\001\end{smallmatrix}$};
            \node[] (l) at (1,2) {\tiny $\begin{smallmatrix}101\\010\end{smallmatrix}$};
            \node[] (m) at (2,2) {\tiny $\begin{smallmatrix}101\\011\end{smallmatrix}$};
            \node[] (n) at (3,2) {\tiny $\begin{smallmatrix}110\\001\end{smallmatrix}$};
            \path [-,green] (0) edge (a);
            \path [-,red] (0) edge (b);
            \path [-,red] (0) edge (c);
            \path [-,black] (0) edge (d);
            \path [-,black] (0) edge (e);
            \path [-,cyan] (0) edge (f);
            \path [-,cyan] (0) edge (g);
            \path [-,lightgray] (1) edge (h);
            \path [-,lightgray] (1) edge (i);
            \path [-,lightgray] (1) edge (j);
            \path [-,lightgray] (1) edge (k);
            \path [-,lightgray] (1) edge (l);
            \path [-,lightgray] (1) edge (m);
            \path [-,lightgray] (1) edge (n);
            \path [-,lightgray] (a) edge (h);
            \path [-,lightgray] (a) edge (i);
            \path [-,lightgray] (a) edge (j);
            \path [-,lightgray] (b) edge (h);
            \path [-,lightgray] (b) edge (k);
            \path [-,lightgray] (b) edge (l);
            \path [-,lightgray] (c) edge (h);
            \path [-,lightgray] (c) edge (m);
            \path [-,lightgray] (c) edge (n);
            \path [-,lightgray] (d) edge (i);
            \path [-,lightgray] (d) edge (l);
            \path [-,lightgray] (d) edge (n);
            \path [-,lightgray] (e) edge (i);
            \path [-,lightgray] (e) edge (k);
            \path [-,lightgray] (e) edge (m);
            \path [-,lightgray] (f) edge (j);
            \path [-,lightgray] (f) edge (k);
            \path [-,lightgray] (f) edge (n);
            \path [-,lightgray] (g) edge (j);
            \path [-,lightgray] (g) edge (l);
            \path [-,lightgray] (g) edge (m);
        \end{tikzpicture}}   & \scalebox{1.1}{
\begin{tikzpicture}
            \node[] (0) at (0,0) {\tiny $0$};
            \node[] (1) at (0,3) {\tiny $\mathbb{F}_2^3$};
            \node[] (a) at (-3,1) {\tiny $100$};
            \node[] (b) at (-2,1) {\tiny $010$};
            \node[] (c) at (-1,1) {\tiny $110$};
            \node[] (d) at (0,1) {\tiny $111$};
            \node[] (e) at (1,1) {\tiny $011$};
            \node[] (f) at (2,1) {\tiny $001$};
            \node[] (g) at (3,1) {\tiny $101$};
            \node[] (h) at (-3,2) {\tiny $\begin{smallmatrix}100\\010\end{smallmatrix}$};
            \node[] (i) at (-2,2) {\tiny $\begin{smallmatrix}100\\011\end{smallmatrix}$};
            \node[] (j) at (-1,2) {\tiny $\begin{smallmatrix}100\\001\end{smallmatrix}$};
            \node[] (k) at (0,2) {\tiny $\begin{smallmatrix}010\\001\end{smallmatrix}$};
            \node[] (l) at (1,2) {\tiny $\begin{smallmatrix}101\\010\end{smallmatrix}$};
            \node[] (m) at (2,2) {\tiny $\begin{smallmatrix}101\\011\end{smallmatrix}$};
            \node[] (n) at (3,2) {\tiny $\begin{smallmatrix}110\\001\end{smallmatrix}$};
            \path [-,green] (0) edge (a);
            \path [-,red] (0) edge (b);
            \path [-,red] (0) edge (c);
            \path [-,black] (0) edge (d);
            \path [-,black] (0) edge (e);
            \path [-,cyan] (0) edge (f);
            \path [-,cyan] (0) edge (g);
            \path [-,black] (1) edge (h);
            \path [-,red] (1) edge (i);
            \path [-,green] (1) edge (j);
            \path [-,green] (1) edge (k);
            \path [-,green] (1) edge (l);
            \path [-,green] (1) edge (m);
            \path [-,green] (1) edge (n);
            \path [-,red] (a) edge (h);
            \path [-,black] (a) edge (i);
            \path [-,cyan] (a) edge (j);
            \path [-,green] (b) edge (h);
            \path [-,black] (b) edge (k);
            \path [-,black] (b) edge (l);
            \path [-,green] (c) edge (h);
            \path [-,black] (c) edge (m);
            \path [-,black] (c) edge (n);
            \path [-,green] (d) edge (i);
            \path [-,red] (d) edge (l);
            \path [-,red] (d) edge (n);
            \path [-,green] (e) edge (i);
            \path [-,red] (e) edge (k);
            \path [-,red] (e) edge (m);
            \path [-,green] (f) edge (j);
            \path [-,green] (f) edge (k);
            \path [-,green] (f) edge (n);
            \path [-,green] (g) edge (j);
            \path [-,green] (g) edge (l);
            \path [-,green] (g) edge (m);
        \end{tikzpicture}} \\
        \end{tabular}
        \caption{A set of atomic weights (\textcolor{green}{0},\textcolor{red}{1}, 2, \textcolor{cyan}{3}) in $\mathcal{L}(\mathbb{F}_2^3)$ and the corresponding $q$-polymatroid.}
        \label{fig:z1ii_counter}
    \end{figure}
    
    Now consider the same lattice $\mathcal{L}(\F_2^3)$ but with the atomic weights as shown in Figure~\ref{fig:z1ii_counter_alt}. For both $\mathcal{Z}_1$ and $\mathcal{Z}_2$, all of the cyclic flat axioms except for \ref{z1}\ref{z1ii} are satisfied. In fact, by inspecting the interval $[0,\langle e_2,e_3\rangle]$, we see that the interval $[\langle e_2+e_3\rangle,\langle e_2,e_3\rangle]$ must have a cover weight of at least two, which exceeds the cover weight of $[0,\langle e_2\rangle]$. This contradicts the cover-weight axioms, implying that there is no $q$-polymatroid with this set of atomic weights. Therefore, there is no lattice of cyclic flats that satisfies all of the cyclic flat axioms with this set of atomic weights.
    
    \begin{figure}[h]
        \centering
        \scalebox{1.5}{
\begin{tikzpicture}
            \node[] (0) at (0,0) {\tiny $0$};
            \node[] (1) at (0,3) {\tiny $\mathbb{F}_2^3$};
            \node[] (a) at (-3,1) {\tiny $100$};
            \node[] (b) at (-2,1) {\tiny $010$};
            \node[] (c) at (-1,1) {\tiny $110$};
            \node[] (d) at (0,1) {\tiny $111$};
            \node[] (e) at (1,1) {\tiny $011$};
            \node[] (f) at (2,1) {\tiny $001$};
            \node[] (g) at (3,1) {\tiny $101$};
            \node[] (h) at (-3,2) {\tiny $\begin{smallmatrix}100\\010\end{smallmatrix}$};
            \node[] (i) at (-2,2) {\tiny $\begin{smallmatrix}100\\011\end{smallmatrix}$};
            \node[] (j) at (-1,2) {\tiny $\begin{smallmatrix}100\\001\end{smallmatrix}$};
            \node[] (k) at (0,2) {\tiny $\begin{smallmatrix}010\\001\end{smallmatrix}$};
            \node[] (l) at (1,2) {\tiny $\begin{smallmatrix}101\\010\end{smallmatrix}$};
            \node[] (m) at (2,2) {\tiny $\begin{smallmatrix}101\\011\end{smallmatrix}$};
            \node[] (n) at (3,2) {\tiny $\begin{smallmatrix}110\\001\end{smallmatrix}$};
            \path [-,green] (0) edge (a);
            \path [-,red] (0) edge (b);
            \path [-,red] (0) edge (c);
            \path [-,black] (0) edge (d);
            \path [-,black] (0) edge (e);
            \path [-,blue] (0) edge (f);
            \path [-,blue] (0) edge (g);
            \path [-,lightgray] (1) edge (h);
            \path [-,lightgray] (1) edge (i);
            \path [-,lightgray] (1) edge (j);
            \path [-,lightgray] (1) edge (k);
            \path [-,lightgray] (1) edge (l);
            \path [-,lightgray] (1) edge (m);
            \path [-,lightgray] (1) edge (n);
            \path [-,lightgray] (a) edge (h);
            \path [-,lightgray] (a) edge (i);
            \path [-,lightgray] (a) edge (j);
            \path [-,lightgray] (b) edge (h);
            \path [-,lightgray] (b) edge (k);
            \path [-,lightgray] (b) edge (l);
            \path [-,lightgray] (c) edge (h);
            \path [-,lightgray] (c) edge (m);
            \path [-,lightgray] (c) edge (n);
            \path [-,lightgray] (d) edge (i);
            \path [-,lightgray] (d) edge (l);
            \path [-,lightgray] (d) edge (n);
            \path [-,lightgray] (e) edge (i);
            \path [-,lightgray] (e) edge (k);
            \path [-,lightgray] (e) edge (m);
            \path [-,lightgray] (f) edge (j);
            \path [-,lightgray] (f) edge (k);
            \path [-,lightgray] (f) edge (n);
            \path [-,lightgray] (g) edge (j);
            \path [-,lightgray] (g) edge (l);
            \path [-,lightgray] (g) edge (m);
        \end{tikzpicture}}
        \caption{A set of atomic weights (\textcolor{green}{0},\textcolor{red}{1},2,\textcolor{blue}{4}) in the lattice $\mathcal{L}(\mathbb{F}_2^3)$.}
        \label{fig:z1ii_counter_alt}
    \end{figure}
\end{example}

\begin{example}
    In the lattice $\mathcal{L}(\mathbb{F}_2^2)$, suppose that $\mathcal{Z}=\{\langle e_1\rangle\}$
    and that the atomic weights are given as in Figure~\ref{fig:z1i_counter}. Thus, all of the cyclic flat axioms except for \ref{z1}\ref{z1i} are satisfied. (By inspection, it is clear that a $q$-polymatroid with such a lattice and set of atomic weights is impossible.)
    
    \begin{figure}[h]
        \centering
        \scalebox{1.5}{
        \begin{tikzpicture}
            \node[] (0) at (0,0) {\tiny $0$};
            \node[] (1) at (0,2) {\tiny $\mathbb{F}_2^2$};
            \node[] (i) at (-1,1) {\tiny $10$};
            \node[] (j) at (0,1) {\tiny $01$};
            \node[] (m) at (1,1) {\tiny $11$};
            \path [-,green] (0) edge (i);
            \path [-,red] (0) edge (j);
            \path [-,black] (0) edge (m);
            \path [-,lightgray] (1) edge (i);
            \path [-,lightgray] (1) edge (j);
            \path [-,lightgray] (1) edge (m);
        \end{tikzpicture}}
        \caption{A set of atomic weights (\textcolor{green}{0}, \textcolor{red}{1}, 2) in the lattice $\mathcal{L}(\mathbb{F}_2^2)$.}
        \label{fig:z1i_counter}
    \end{figure}
\end{example}

\section*{Acknowledgement}
This publication has emanated from research
conducted with the financial support of Science Foundation Ireland under grant
number 18/CRT/6049.

\section*{Data Availability Statement}
Data sharing not applicable to this article as no datasets were generated or analysed during the current study.

\section*{Conflict of Interest Statement}
On behalf of all authors, the corresponding author states that there is no conflict of interest.

\newpage
\bibliographystyle{abbrv}
\bibliography{References} 

\end{document}